\documentclass[a4paper,11pt]{article}
\usepackage[utf8]{inputenc}
\usepackage[OT1,T1]{fontenc}
\usepackage{amsmath,amssymb}
\usepackage{mathrsfs}
\usepackage{bm}
\usepackage{upgreek}
\usepackage{tikz}
\usepackage[colorlinks,citecolor=black,linkcolor=black,urlcolor=black,
  pdfpagemode=None,pdfstartview=FitH]{hyperref}

\newtheorem{theorem}{Theorem}
\newtheorem{lemma}[theorem]{Lemma}
\newtheorem{proposition}[theorem]{Proposition}

\newenvironment{proof}{\trivlist
  \item[\hskip\labelsep{\itshape Proof.}]\upshape}{\nobreak\noindent
  $\square$\endtrivlist}
\newenvironment{other}[1]{\refstepcounter{theorem}\trivlist
  \item[\hskip\labelsep{\itshape #1~\arabic{theorem}.}]
  \upshape}{\endtrivlist\bigbreak}
\newenvironment{other*}[1]{\trivlist
  \item[\hskip\labelsep{\itshape #1.}]
  \upshape}{\endtrivlist\bigbreak}

\renewcommand\theenumi{(\roman{enumi})}

\newcommand\wt{{\mathop{\mathrm{wt}\,}}}

\newcommand\SL{\mathord{\mathrm{SL}}}
\newcommand\SO{\mathord{\mathrm{SO}}}

\newcommand\PGL{\mathord{\mathrm{PGL}}}
\newcommand\Gr{\mathord{\mathrm{Gr}}}

\newcommand\BM{{\mathrm{BM}}}
\newcommand\BDConv{\mathord{\mathcal Gr}}
\newcommand\caltimes{\mathbin{\propto}}

\begin{document}
\title{Mirković-Vilonen basis in type $A_1$}
\author{Pierre Baumann and Arnaud Demarais}
\date{}
\maketitle

\begin{abstract}
\noindent
Let $G$ be a reductive connected algebraic group over $\mathbb C$.
Through the geometric Satake equivalence, the fundamental classes of
the Mirković-Vilonen cycles define a basis in each tensor product
$V(\lambda_1)\otimes\cdots\otimes V(\lambda_n)$ of irreducible
representations of $G$. In the case $G=\SL_2(\mathbb C)$, we show
that this basis coincides with the dual canonical basis at $q=1$.
\end{abstract}

\section{Introduction}
Let $G$ be a reductive connected algebraic group over $\mathbb C$,
endowed with a Borel subgroup $B$ and a maximal torus $T\subset B$.
Irreducible rational representations of $G$ are classified by their
highest weight: to the dominant integral weight $\lambda$ corresponds
the irreducible representation~$V(\lambda)$.

Several constructions allow to define nice bases of $V(\lambda)$, for
instance:
\begin{itemize}
\item
From the study of quantum groups, Lusztig \cite{Lusztig90} defined his
canonical basis in the quantum deformation $V_q(\lambda)$; taking the
classical limit $q=1$ provides a basis of $V(\lambda)$. For convenience,
we will in fact use the dual canonical of this basis, aka Kashiwara's
upper global basis~\cite{Kashiwara}.
\item
The geometric Satake correspondence \cite{Lusztig81,MirkovicVilonen}
realizes $V(\lambda)$ as the intersection cohomology of certain Schubert
varieties $\overline{\Gr^\lambda}$ in the affine Grassmannian of the
Langlands dual of $G$. The fundamental classes of the Mirković-Vilonen
cycles form a basis of this cohomology space, hence of $V(\lambda)$.
\end{itemize}

These two constructions can be extended to tensor products
$V(\lambda_1)\otimes\cdots\otimes V(\lambda_r)$, see chapter~27 in
\cite{Lusztig93} for the former and sect.~2.4 in \cite{GoncharovShen} for
the latter. These two bases share several nice properties, for instance
both are compatible with the isotypical filtration and with restriction
to standard Levi subgroups; also both are difficult to compute.
In general they differ: an example with $G=\SL_3(\mathbb C)$ and $r=12$
is given in~\cite{FontaineKamnitzerKuperberg}; examples for $r=1$
(hence for irreducible representations) are given in
\cite{BaumannKamnitzerKnutson} for $G=\SO_8(\mathbb C)$ and
$G=\SL_6(\mathbb C)$.

In type $A_1$, that is for $G=\SL_2(\mathbb C)$, the dual canonical basis
was computed by Frenkel and Khovanov~\cite{FrenkelKhovanov}. The aim of
this paper is to do the analog for the Mirković-Vilonen basis.

\trivlist
\item[\hskip\labelsep{\bfseries Theorem.}]
\itshape
For $G=\SL_2(\mathbb C)$, the Mirković-Vilonen basis of a tensor product
$V(\lambda_1)\otimes\cdots\otimes V(\lambda_r)$ coincides with the
dual canonical basis of this space specialized at $q=1$.
\upshape
\endtrivlist

This result is trivial in the case $r=1$ of an irreducible representation,
but the general case seems less obvious. We must also point out that in
truth, this result holds only after reversal of the order of the tensor
factors, but this defect is merely caused by a difference in the conventions.

In this case $G=\SL_2(\mathbb C)$, each dominant weight is a nonnegative
multiple of the fundamental weight $\varpi$. Then $V(n\varpi)$ has
dimension $n+1$ and is the Cartan component, i.e.\ the top step in the
isotypical filtration, of $V(\varpi)^{\otimes n}$. We can thus regard
$V(n_1\varpi)\otimes\cdots\otimes V(n_r\varpi)$ as a quotient of
$V(\varpi)^{\otimes(n_1+\cdots+n_r)}$. Since both the dual canonical
basis and the Mirković-Vilonen basis behave well under this quotient
operation, it is enough to establish the theorem in the particular case
of the tensor power $V(\varpi)^{\otimes n}$.

This paper is organized in the following way. In sect.~\ref{se:CombLin},
we define a basis of $V(\varpi)^{\otimes n}$ by a simple recursive formula
and argue that it matches Frenkel and Khovanov's characterization of the
dual canonical basis. In sect.~\ref{se:MVBasis}, we recall the definition of
the Mirković-Vilonen basis in tensor products of irreducible representations
and prove its good behavior under the quotient operation mentioned in
the previous paragraph. In sect.~\ref{se:Geometry}, we show that the
Mirković-Vilonen basis of $V(\varpi)^{\otimes n}$ satisfies the recursive
formula from sect.~\ref{se:CombLin} (this is the difficult part in the paper).

This work is based on the PhD thesis of the second author~\cite{Demarais}.
We however rewrote the proof to render it more accessible and remove
ambiguities.

While readying this paper, we learned that independently Pak-Hin Li
computed the Mirković-Vilonen basis for the tensor product of two
irreducible representations of $\SL_2(\mathbb C)$.

\textit{Acknowledgements.}
P.B.'s research is supported by the ANR project GeoLie,
ANR-15-CE40-0012.

\section{Combinatorics and linear algebra}
\label{se:CombLin}
Let $\mathbb K$ be a field and let $V$ be the vector space $\mathbb K^2$.
In this section, we define in an elementary manner an explicit basis in
each tensor power $V^{\otimes n}$ that has nice properties with respect
to the natural action of $\SL_2(\mathbb K)$.

\subsection{Words}
\label{ss:Words}
Given a nonnegative integer $n$, we set $\mathscr C_n=\{+,-\}^n$.
We regard an element in $\mathscr C_n$ as a word of length $n$ on the
alphabet $\{+,-\}$. Concatenation of words endows
$\mathscr C=\bigcup_{n\geq0}\mathscr C_n$ with the structure of a monoid.

The weight of a word $w\in\mathscr C$, denoted by $\wt(w)$, is the number
of letters $+$ minus the number of letters $-$ that $w$ contains. A word
$w=w(1)w(2)\cdots w(n)$ is said to be semistable if its weight is $0$ and
if each initial segment $w(1)\cdots w(j)$ has nonpositive weight.

Words are best understood through a representation as planar paths,
where letters $+$ and $-$ are depicted by upward and downward segments,
respectively. A word is semistable if the endpoints of its graphical
representation are on the same horizontal line and if the whole path
lies below this line.

Any word $w$ can be uniquely factorized as a concatenation
$$w_{-r}+\cdots+w_{-1}+w_0-w_1-\cdots-w_s$$
where $r$ and $s$ are nonnegative integers and where the words
$w_{-r}$, \dots, $w_s$ are semistable. The $r$ letters $+$ and
the $s$ letters $-$ that do not occur in the semistable words are
called significant; informally, a letter $+$ is significant if it
marks the first time a new highest altitude is reached.

\begin{other*}{Example}
The following picture illustrates the factorization of the word
$$w=\color{orange}-+\color{black}+\color{orange}-+-+\color{black}++
\color{orange}--+--+++\color{black}+-\color{orange}-+\color{black}-.$$

\vspace*{-18pt}
This word has length $22$ and weight $2$. Here $(r,s)=(4,2)$ and the
words $w_{-2}$, $w_0$ and $w_2$ are empty. Significant letters are
written in black.
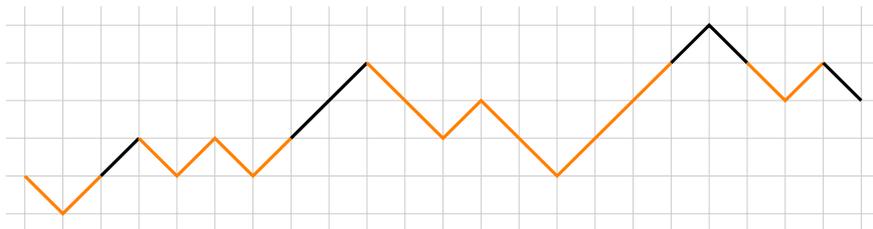
\begin{center}
\begin{tikzpicture}[scale=0.5]
\draw[very thin,color=gray!40] (-0.5,-1.5) grid (22.5,4.5);
\draw[very thick,orange] (0,0)--(1,-1)--(2,0);
\draw[very thick] (2,0)--(3,1);
\draw[very thick,orange] (3,1)--(4,0)--(5,1)--(6,0)--(7,1);
\draw[very thick] (7,1)--(9,3);
\draw[very thick,orange] (9,3)--(10,2)--(11,1)--(12,2)--(13,1)--(14,0)--(15,1)--
  (16,2)--(17,3);
\draw[very thick] (17,3)--(18,4)--(19,3);
\draw[very thick,orange] (19,3)--(20,2)--(21,3);
\draw[very thick] (21,3)--(22,2);
\end{tikzpicture}
\end{center}
\end{other*}

Given a word $w$, we denote by $\mathscr P(w)$ the set of words obtained
from $w$ by changing a single significant letter $+$ into a $-$. With our
previous notation, $\mathscr P(w)$ has $r$ elements.

\subsection{Bases}
\label{ss:Bases}
Let $(x_+,x_-)$ be the standard basis of the vector space $V$.
Each word $w=w(1)w(2)\cdots w(n)$ in $\mathscr C^n$ defines an element
$x_w=x_{w(1)}\otimes\cdots\otimes x_{w(n)}$ in the $n$-th tensor power of
$V$. The family $(x_w)_{w\in\mathscr C_n}$ is a basis of $V^{\otimes n}$.

We define another family of elements $(y_w)_{w\in\mathscr C}$ in
the tensor algebra of $V$ by the convention $y_\varnothing=1$ and the
recursive formulas
$$y_{+w}=x_+\otimes y_w\quad\text{and}\quad
y_{-w}=x_-\otimes y_w-\sum_{v\in\mathscr P(w)}x_+\otimes y_v.$$

Rewritting the latter as
\begin{equation}
\label{eq:DefYw}
x_+\otimes y_w=y_{+w}\quad\text{and}\quad
x_-\otimes y_w=y_{-w}+\sum_{v\in\mathscr P(w)}y_{+v}
\end{equation}
one easily shows by induction on the length of words that each element
$x_w$ can be expressed as a linear combination of elements $y_v$, with
$v$ in the set of words that the same length and weight as~$w$. It
follows in particular that for each nonnegative integer $n$, the family
$(y_w)_{w\in\mathscr C_n}$ spans $V^{\otimes n}$, hence is a basis of
this space.

\begin{proposition}
\label{pr:CaracYw}
The family $(y_w)_{w\in\mathscr C}$ is characterized by the following
conditions:

\vspace{-12pt}
\begin{enumerate}
\item
\label{it:PrCYa}
If $w$ is of the form $+\cdots+-\cdots-$, then $y_w=x_w$.
\item
\label{it:PrCYb}
$y_{-+}=x_{-+}-x_{+-}$.
\item
\label{it:PrCYc}
Let $u$ be a semistable word and let $(w',w'')\in\mathscr C_{n'}
\times\mathscr C_{n''}$. Write
$y_{w'w''}=\sum\limits_i\;y'_i\otimes y''_i$ with
$(y'_i,y''_i)\in V^{\otimes n'}\times V^{\otimes n''}$.
Then $y_{w'uw''}=\sum\limits_i\;y'_i\otimes y_u\otimes y''_i$.
\iffalse
Let $(u,w',w'')\in\mathscr C_m\times\mathscr C_{n'}\times\mathscr C_{n''}$.
Define a linear map
$J:V^{\otimes n'}\otimes V^{\otimes n''}\to V^{\otimes n'+m+n''}$ by
$J(x'\otimes x'')=x'\otimes y_u\otimes x''$.
If $u$ is semistable, then $y_{w'uw''}=J(y_{w'w''})$.
\fi
\end{enumerate}
\end{proposition}
\begin{proof}
Statements~\ref{it:PrCYa} and~\ref{it:PrCYb} follow straightforwardly
from the definition of the elements $y_w$. We prove~\ref{it:PrCYc} by
induction on the length of $w'uw''$. Discarding a trivial case, we
assume that $u$ is not the empty word.

Suppose first that $w'$ is the empty word. Let us write $u$ as a
concatenation $-u'+u''$ where $u'$ and $u''$ are (possibly empty)
semistable words. Equation~\eqref{eq:DefYw} gives
$$x_-\otimes y_{w''}=y_{-w''}+\sum_{v\in\mathscr P(w'')}y_{+v}.$$
Applying the induction hypothesis to the semistable word $u''$ and the
pairs $(-,w'')$ and $(+,v)$, for each $v\in\mathscr P(w'')$, we obtain
$$x_-\otimes y_{u''}\otimes y_{w''}=
y_{-u''w''}+\sum_{v\in\mathscr P(w'')}y_{+u''v}.$$
Since $x_-\otimes y_{u''}=y_{-u''}$, we get
$$y_{-u''}\otimes y_{w''}=y_{-u''w''}+\sum_{v\in\mathscr P(w'')}y_{+u''v}$$
and applying once more the induction hypothesis, this time to the
semistable word $u'$ and the pairs $(\varnothing,-u'')$,
$(\varnothing,-u''w'')$ and $(\varnothing,+u''v)$, we arrive at
\begin{equation}
\label{eq:PrCY1}
y_{u'-u''}\otimes y_{w''}=y_{u'-u''w''}
+\sum_{v\in\mathscr P(w'')}y_{u'+u''v}.
\end{equation}
Starting now with
$$x_+\otimes y_{w''}=y_{+w''}$$
we arrive by similar transformations at
\begin{equation}
\label{eq:PrCY2}
y_{u'+u''}\otimes y_{w''}=y_{u'+u''w''}.
\end{equation}
Since $\mathscr P(u'+u'')=\{u'-u''\}$, we have by definition
\begin{equation}
\label{eq:PrCY3}
y_u=x_-\otimes y_{u'+u''}-x_+\otimes y_{u'-u''}.
\end{equation}
Likewise, $\mathscr P(u'+u''w'')=\{u'-u''w''\}\cup\{u'+u''v\mid
v\in\mathscr P(w'')\}$ leads to
\begin{equation}
\label{eq:PrCY4}
y_{uw''}=x_-\otimes y_{u'+u''w''}-x_+\otimes y_{u'-u''w''}
-\sum_{v\in\mathscr P(w'')}x_+\otimes y_{u'+u''v}.
\end{equation}
Combining \eqref{eq:PrCY1}--\eqref{eq:PrCY4}, we obtain the desired equation
$$y_{uw''}=y_u\otimes y_{w''}.$$

We now address the case where $w'$ is not empty. Suppose that the
first letter of $w'$ is a $+$ and write $w'=+\,\widetilde w'$. Then
$$y_{w'w''}=x_+\otimes y_{\widetilde w'w''}\quad\text{and}\quad
y_{w'uw''}=x_+\otimes y_{\widetilde w'uw''}$$
and the result follows from the induction hypothesis applied to the
semistable word $u$ and the pair $(\widetilde w',w'')$.

If on the contrary the first letter of $w'$ is a $-$, then we write
$w'=-\,\widetilde w'$. Since $u$ is semistable, its insertion in the
middle of a word does not add or remove any significant letter; in
particular, the set of significant letters in $\widetilde w'w''$ is
in natural bijection with the set of significant letters in
$\widetilde w'uw''$. This observation leads to a bijection from
$\mathscr P(\widetilde w'w'')$ onto $\mathscr P(\widetilde w'uw'')$,
which splits a word $v$ in two subwords $v'\in\mathscr C_{n'-1}$ and
$v''\in\mathscr C_{n''}$ and then returns $v'uv''$. With this notation,
$$y_{w'w''}=x_-\otimes y_{\widetilde w'w''}-
\sum_{v\in\mathscr P(\widetilde w'w'')}x_+\otimes y_{v'v''}$$
and
$$y_{w'uw''}=x_-\otimes y_{\widetilde w'uw''}-
\sum_{v\in\mathscr P(\widetilde w'w'')}x_+\otimes y_{v'uv''}.$$
Again the desired equation follows from the induction hypothesis
applied to the semistable word $u$ and the pairs $(\widetilde w',w'')$
and $(v',v'')$, for each $v\in\mathscr P(\widetilde w'w'')$.

Condition~\ref{it:PrCYc} computes $y_{w'uw''}$ from the datum of
$y_{w'w''}$ and $y_u$ whenever $u$ is semistable;
condition~\ref{it:PrCYa} provide the value of $y_w$ when $w$ is of
the form $+\cdots+-\cdots-$; and condition~\ref{it:PrCYb} provides
the value of $y_{-+}$. Noting that any word in $\mathscr C$ can
be obtained from a word of the form $+\cdots+-\cdots-$ by repetitively
inserting the semistable word $-+$ (possibly at non disjoint positions),
we conclude that conditions \ref{it:PrCYa}--\ref{it:PrCYc} fully
characterize the family $(y_w)_{w\in\mathscr C}$.
\end{proof}

As a consequence of this proposition, we see that if
$w_{-k}+\cdots+w_{-1}+w_0-w_1-\cdots-w_\ell$ is the factorization of a
word $w$, as in section~\ref{ss:Words}, then
\begin{equation}
\label{eq:FacYw}
y_w=y_{w_{-k}}\otimes x_+\otimes\cdots\otimes x_+\otimes
y_{w_{-1}}\otimes x_+\otimes y_{w_0}\otimes x_-\otimes y_{w_1}\otimes
x_-\otimes \cdots\otimes x_-\otimes y_{w_\ell}.
\end{equation}

\begin{other*}{Remark}
The transition matrix between the two bases $(x_w)_{w\in\mathscr C_n}$
and $(y_w)_{w\in\mathscr C_n}$ of $V^{\otimes n}$ is unitriangular: if
we write
$$x_w=\sum_{v\in\mathscr C_n}n_{w,v}\,y_v$$
then the diagonal coefficient $n_{w,w}$ is equal to one and the
entry $n_{w,v}$ is zero except when the path representing $v$ lies
above the path representing $w$. In addition, all the coefficients
$n_{w,v}$ are nonnegative integers. The proof of these facts is
left to the reader.
\end{other*}

\subsection{Representations}
\label{ss:Reps}
In this section, we regard $V$ as the defining representation of
$\SL_2(\mathbb K)$. From now on, we assume that $\mathbb K$ has
characteristic zero. We denote by $(e,h,f)$ the usual basis of
$\mathfrak{sl}_2(\mathbb K)$.

Fix a nonnegative integer $n$. Given a word $w\in\mathscr C_n$,
we denote by $\varepsilon(w)$ (respectively, $\varphi(w)$) the number
of significant letters $-$ (respectively, $+$) that $w$ contains.
Thus, in the notation of section \ref{ss:Words}, $\varepsilon(w)=s$
and $\varphi(w)=r$. If $\varepsilon(w)>0$, we can change in $w$ the
leftmost significant letter $-$ into a $+$; the resulting word is
denoted by $\tilde e(w)$. Likewise, if $\varphi(w)>0$, we can change
in $w$ the rightmost significant letter $+$ into a $-$; the resulting
word is denoted by $\tilde f(w)$. If these operations are not feasible,
then $\tilde e(w)$ or $\tilde f(w)$ is defined to be $0$. Endowed with
the maps $\wt$, $\varepsilon$, $\varphi$, $\tilde e$, $\tilde f$, the
set $\mathscr C_n$ identifies with the crystal\footnote{In fact, we
here use the opposite of the usual tensor product of crystals.} of the
$\mathfrak{sl}_2(\mathbb K)$-module~$V^{\otimes n}$.

We denote by $\ell(w)=\varepsilon(w)+\varphi(w)$ the number of significant
letters in a word $w\in\mathscr C_n$; thus $w$ is semistable if and only
if $\ell(w)=0$. For each $p\in\{0,\ldots,n\}$, we denote by
$(V^{\otimes n})_{\leq p}$ the subspace of $V^{\otimes n}$
spanned by the elements $y_w$ such that $\ell(w)\leq p$. We agree that
$(V^{\otimes n})_{\leq-1}=\{0\}$.

\begin{proposition}
\label{pr:RepYw}
The basis $(y_w)_{w\in\mathscr C_n}$ of $V^{\otimes n}$ enjoys the
following properties.

\vspace{-12pt}
\begin{enumerate}
\item
\label{it:PrRYa}
For each $w\in\mathscr C_n$, we have
$$e\cdot y_w\equiv\varepsilon(w)\;y_{\tilde e(w)}\quad\text{and}\quad
f\cdot y_w\equiv\varphi(w)\;y_{\tilde f(w)}$$
modulo terms in $(V^{\otimes n})_{\leq\ell(w)-1}$.
\item
\label{it:PrRYb}
For each $p\in\{0,\ldots,n\}$, the subspace $(V^{\otimes n})_{\leq p}$
is a subrepresentation of $V^{\otimes n}$, and the quotient
$(V^{\otimes n})_{\leq p}/(V^{\otimes n})_{\leq p-1}$ is an isotypical
representation, sum of simple $\mathfrak{sl}_2(\mathbb K)$-modules of
dimension $p+1$.
\item
\label{it:PrRYc}
The elements $y_w$ with $w$ semistable form a basis of the space of
invariants $(V^{\otimes n})^{\SL_2(\mathbb K)}$.
\end{enumerate}
\end{proposition}
\trivlist
\item[\hskip\labelsep{\itshape Sketch of proof.}]
\upshape
We first note that any semistable word can be obtained from the empty
word by repetitively inserting the word $-+$ and that $y_{-+}$ is
invariant under the action of $\SL_2(\mathbb K)$ on $V^{\otimes2}$.
From Proposition~\ref{pr:CaracYw}~\ref{it:PrCYc}, it then follows that
any element $y_w$ with $w$ semistable is $\SL_2(\mathbb K)$-invariant.
Using now \eqref{eq:FacYw}, we reduce the proof of statement~\ref{it:PrRYa}
to the case where $w$ is of the form $+\cdots+-\cdots-$ (though possibly
for a smaller $n$), which is easily dealt with.

Statement~\ref{it:PrRYb} is a direct consequence of
statement~\ref{it:PrRYa} and implies that $(V^{\otimes n})_{\leq0}$
is the subspace of invariants $(V^{\otimes n})^{\SL_2(\mathbb K)}$,
an assertion equivalent to statement~\ref{it:PrRYc}.
\nobreak$\square$
\endtrivlist

The basis $(y_w)_{w\in\mathscr C_n}$ of~$V^{\otimes n}$ is even
more remarkable than what Proposition~\ref{pr:RepYw} claims.
Comparing Frenkel and Khovanov's work (\cite{FrenkelKhovanov},
Theorem~1.9) with Proposition~\ref{pr:CaracYw}, we indeed get:

\begin{theorem}
\label{th:FrenKhov}
$(y_w)_{w\in\mathcal C_n}$ is the dual canonical basis of
$V^{\otimes n}$ specialized at $q=1$.
\end{theorem}

As mentioned in the introduction, this result actually holds only after
reversal of the order of the tensor factors.

\section{The Mirković-Vilonen basis}
\label{se:MVBasis}
In this section, we consider a connected reductive group $G$ over
$\mathbb C$ and explain the definition of the Mirković-Vilonen basis
(from now on: MV basis) in a tensor product
$V(\bm\lambda)=V(\lambda_1)\otimes\cdots\otimes V(\lambda_n)$ of
irreducible representations of $G$; references for the material
presented here are \cite{MirkovicVilonen} and sect.~2.4 in
\cite{GoncharovShen}. We recall the recipe from
\cite{BaumannGaussentLittelmann} to compute the transition matrix
between the MV basis of $V(\bm\lambda)$ and the tensor product of
the MV bases of the factors $V(\lambda_1)$, \dots, $V(\lambda_n)$.
We state and prove a compatibility property of the MV bases with
tensor products of projections onto Cartan components.

\subsection{Definition of the basis}
\label{ss:DefBasis}
We choose a maximal torus $T$ and a Borel subgroup $B$ of $G$ such that
$T\subset B$. The Langlands dual $G^\vee$ of $G$ comes with a maximal
torus $T^\vee$ and a Borel subgroup $B^\vee$. We denote by $N^{-,\vee}$
the unipotent radical of the Borel subgroup of $G^\vee$ opposite to
$B^\vee$ with respect to $T^\vee$. We denote by $\Lambda$ the weight
lattice of $T$ and by $\Lambda^+\subset\Lambda$ the set of dominant
weights. Let $\leq$ be the dominance order on $\Lambda$; positive elements
with respect to $\leq$ are sums of positive roots. We view the half-sum
of all positive coroots as a linear form $\rho:\Lambda\to\mathbb Q$.

The affine Grassmannian of $G^\vee$ is the homogeneous space
$\Gr=G^\vee\bigl(\mathbb C\bigl[z,z^{-1}\bigr]\bigr)/
G^\vee(\mathbb C[z])$, where $z$ is an indeterminate.
It is endowed with the structure of an ind-variety.

Each weight $\lambda\in\Lambda$ gives a point $z^\lambda$ in
$T^\vee\bigl(\mathbb C\bigl[z,z^{-1}\bigr]\bigr)$, whose image in $\Gr$
is denoted by $L_\lambda$. The $G^\vee(\mathbb C[z])$-orbit through
$L_\lambda$ in $\Gr$ is denoted by $\Gr^\lambda$; this is a smooth
connected variety of dimension $2\rho(\lambda)$. The Cartan decomposition
implies that
$$\Gr=\bigsqcup_{\lambda\in\Lambda^+}\Gr^\lambda;\quad\text{moreover}\quad
\overline{\Gr^\lambda}=\bigsqcup_{\substack{\mu\in\Lambda^+\\[2pt]
\mu\leq\lambda}}\Gr^\mu.$$
The geometric Satake correspondence identifies the irreducible
representation of highest weight $\lambda$ with the intersection
cohomology of $\overline{\Gr^\lambda}$ with trivial local system of
coefficients:
$$V(\lambda)=IH\Bigl(\overline{\Gr^\lambda},\underline{\mathbb C}\Bigr).$$

Let $n$ be a positive integer. The group $G^\vee(\mathbb C[z])^n$ acts on
the space $G^\vee\bigl(\mathbb C\bigl[z,z^{-1}\bigr]\bigr)^n$ by
$$(h_1,\ldots,h_n)\cdot(g_1,\ldots,g_n)=(g_1h_1^{-1},h_1g_2h_2^{-1},
\ldots,h_{n-1}g_nh_n^{-1})$$
where $(h_1,\ldots,h_n)\in G^\vee(\mathbb C[z])^n$ and
$(g_1,\ldots,g_n)\in G^\vee\bigl(\mathbb C\bigl[z,z^{-1}\bigr]\bigr)^n$.
The quotient is called the $n$-fold convolution variety and is denoted
by $\Gr_n$. We will use the customary notation
$$\Gr_n=G^\vee\bigl(\mathbb C\bigl[z,z^{-1}\bigr]\bigr)\,
\times^{G^\vee(\mathbb C[z])}\;\cdots\,\times^{G^\vee(\mathbb C[z])}\;
G^\vee\bigl(\mathbb C\bigl[z,z^{-1}\bigr]\bigr)\,/\,G^\vee(\mathbb C[z])$$
to indicate this construction and denote the image in $\Gr_n$ of a tuple
$(g_1,\ldots,g_n)$ by $[g_1,\ldots,g_n]$. Then $\Gr_n$ is endowed with the
structure of an ind-variety. One may note that $\Gr_1$ is just the affine
Grassmannian $\Gr$. We define a map $m_n:\Gr_n\to\Gr$ by
$m_n([g_1,\ldots,g_n])=[g_1\ldots g_n]$.

For each tuple $\bm\lambda=(\lambda_1,\ldots,\lambda_n)$ in $\Lambda^n$,
we set
$$|\bm\lambda|=\lambda_1+\cdots+\lambda_n.$$
For each $G^\vee(\mathbb C[z])$-invariant subset $K\subset\Gr$, we
denote by $\widehat K$ the preimage of $K$ under the quotient map
$G^\vee\bigl(\mathbb C\bigl[z,z^{-1}\bigr]\bigr)\to\Gr$.
Given $\bm\lambda=(\lambda_1,\ldots,\lambda_n)$ in
$(\Lambda^+)^n$, we define
$$\Gr_n^{\bm\lambda}=\widehat{\Gr^{\lambda_1}}\,
\times^{G^\vee(\mathbb C[z])}\;\cdots\,
\times^{G^\vee(\mathbb C[z])}\;
\widehat{\Gr^{\lambda_n}}\,/\,G^\vee(\mathbb C[z]).$$
This is a subset of $\Gr_n$ and the geometric Satake correspondence
identifies the tensor product
$$V(\bm\lambda)=V(\lambda_1)\otimes\cdots\otimes V(\lambda_n)$$
with the intersection cohomology of $\overline{\Gr_n^{\bm\lambda}}$.

Given $\mu\in\Lambda$, the
$N^{-,\vee}\bigl(\mathbb C\bigl[z,z^{-1}\bigr]\bigr)$-orbit through
$L_\mu$ is denoted by $T_\mu$; this is a locally closed sub-ind-variety
of $\Gr$. The Iwasawa decomposition implies that
$$\Gr=\bigsqcup_{\mu\in\Lambda}T_\mu;\quad\text{moreover}\quad
\overline{T_\mu}=\bigsqcup_{\substack{\nu\in\Lambda\\[2pt]
\nu\geq\mu}}T_\nu.$$
For each $(\lambda,\mu)\in\Lambda^+\times\Lambda$, the intersection
$\overline{\Gr^\lambda}\cap T_\mu$ (if non-empty) has pure dimension
$\rho(\lambda-\mu)$. Using this fact, Mirković and Vilonen set up the
geometric Satake correspondence so that the $\mu$-subspace of
$V(\lambda)$ identifies with the top-dimensional Borel-Moore homology
of $\Gr^\lambda\cap T_\mu$ (\cite{MirkovicVilonen},~Corollary 7.4):
$$V(\lambda)_\mu=H^\BM_{2\rho(\lambda-\mu)}
\Bigl(\Gr^\lambda\cap T_\mu\Bigr).$$
We denote by $\mathscr Z(\lambda)_\mu$ the set of irreducible components
of $\overline{\Gr^\lambda}\cap T_\mu$. If $Z\in\mathscr Z(\lambda)_\mu$,
then $Z\cap\Gr^\lambda$ is an irreducible component of
$\Gr^\lambda\cap T_\mu$, whose fundamental class in Borel-Moore homology
is denoted by $\langle Z\rangle$. The classes $\langle Z\rangle$, for
$Z\in\mathscr Z(\lambda)_\mu$, form a basis of $V(\lambda)_\mu$.

Likewise, for each $(\bm\lambda,\mu)\in(\Lambda^+)^n
\times\Lambda$, the intersection $\overline{\Gr_n^{\bm\lambda}}\cap
(m_n)^{-1}(T_\mu)$ has pure dimension $\rho(|\bm\lambda|-\mu)$.
We can then identify
$$V(\bm\lambda)_\mu=H^\BM_{2\rho(|\bm\lambda|-\mu)}
\Bigl(\Gr_n^{\bm\lambda}\cap(m_n)^{-1}(T_\mu)\Bigr).$$
We denote by $\mathscr Z(\bm\lambda)_\mu$ the set of irreducible
components of $\overline{\Gr_n^{\bm\lambda}}\cap(m_n)^{-1}(T_\mu)$.
If $\mathbf Z\in\mathscr Z(\bm\lambda)_\mu$, then
$\mathbf Z\cap\Gr_n^{\bm\lambda}$ is an irreducible component of
$\Gr_n^{\bm\lambda}\cap(m_n)^{-1}(T_\mu)$, whose fundamental class in
Borel-Moore homology is denoted by $\langle\mathbf Z\rangle$. These classes
$\langle\mathbf Z\rangle$, for $\mathbf Z\in\mathscr Z(\bm\lambda)_\mu$,
form a basis of~$V(\bm\lambda)_\mu$.

We set
$$\mathscr Z(\lambda)=\bigsqcup_{\mu\in\Lambda}\mathscr Z(\lambda)_\mu
\quad\text{and}\quad\mathscr Z(\bm\lambda)=\bigsqcup_{\mu\in\Lambda}
\mathscr Z(\bm\lambda)_\mu.$$
Elements in these sets are called Mirković-Vilonen (MV) cycles, and the
bases of $V(\lambda)$ and $V(\bm\lambda)$ obtained above are called MV
bases.

\subsection{Indexation of the Mirković-Vilonen cycles}
In this short section, we explain that there is a natural bijection
\begin{equation}
\label{eq:ProdCycMV}
\mathscr Z(\bm\lambda)\cong
\mathscr Z(\lambda_1)\times\cdots\times\mathscr Z(\lambda_n)
\end{equation}
for any $\bm\lambda=(\lambda_1,\ldots,\lambda_n)$ in $\Lambda^n$.
The construction goes back to Braverman and
Gaitsgory~\cite{BravermanGaitsgory}; details can be found in
\cite{BaumannGaussentLittelmann}, Proposition~2.2 and Corollary~4.8.

For $\mu\in\Lambda$, we define
$$\widetilde T_\mu=
N^{-,\vee}\bigl(\mathbb C\bigl[z,z^{-1}\bigr]\bigr)\,z^\mu$$
and note that the natural map
$$\widetilde T_\mu\,/\,N^{-,\vee}(\mathbb C[z])\to T_\mu$$
is bijective. Given a $N^{-,\vee}(\mathbb C[z])$-invariant subset
$Z\subset T_\mu$, we denote by $\widetilde Z$ the preimage of $Z$ by
the quotient map $\widetilde T_\mu\to T_\mu$. In particular, the
notation $\widetilde Z$ is defined for any MV cycle $Z$.

Pick $\bm\mu=(\mu_1,\ldots,\mu_n)$ in $\Lambda^n$ and
$\mathbf Z=(Z_1,\ldots,Z_n)$ in
$\mathscr Z(\lambda_1)_{\mu_1}\times\cdots\times
\mathscr Z(\lambda_n)_{\mu_n}$. Then
$$\Bigl\{[g_1,\ldots,g_n]\Bigm|(g_1,\ldots,g_n)\in
\widetilde Z_1\times\cdots\times\widetilde Z_n\Bigr\}$$
is contained in $(m_n)^{-1}\bigl(T_{|\bm\mu|}\bigr)$ and its closure
in this set is an MV cycle in $\mathscr Z(\bm\lambda)_{|\bm\mu|}$.
Each MV cycle in $\mathscr Z(\bm\lambda)$ can be uniquely obtained in
this manner, which defines the bijection~\eqref{eq:ProdCycMV}.

Because of this, we will allow ourselves to write elements in
$\mathscr Z(\bm\lambda)$ as tuples $\mathbf Z$ as above.

\subsection{Transition matrix}
\label{ss:TransMat}
We continue with our tuple of dominant weights
$\bm\lambda=(\lambda_1,\ldots,\lambda_n)$. To compute the MV basis of
$V(\bm\lambda)$, we compare it with the tensor product of the MV bases
of the factors $V(\lambda_1)$, \ldots, $V(\lambda_n)$. This requires the
introduction of a nice geometric object.

Let $n$ be a positive integer. We define the $n$-fold Beilinson-Drinfeld
convolution variety $\BDConv_n$ as the set of pairs
$(x_1,\ldots,x_n;[g_1,\ldots,g_n])$, where
$(x_1,\ldots,x_n)\in\mathbb C^n$ and $[g_1,\ldots,g_n]$ belongs to
$$G^\vee\bigl(\mathbb C\bigl[z,(z-x_1)^{-1}\bigr]\bigr)
\,\times^{G^\vee(\mathbb C[z])}\;\cdots\,\times^{G^\vee(\mathbb C[z])}\;
G^\vee\bigl(\mathbb C\bigl[z,(z-x_n)^{-1}\bigr]\bigr)
\,/\,G^\vee\bigl(\mathbb C[z]\bigr).$$
We denote by $\pi:\BDConv_n\to\mathbb C^n$ the morphism which forgets
$[g_1,\ldots,g_n]$. It is known that $\BDConv_n$ is endowed with
the structure of an ind-variety and that $\pi$ is ind-proper.

To each composition $\mathbf n=(n_1,\ldots,n_r)$ of $n$ in $r$ parts
corresponds a partial diagonal $\Delta_{\mathbf n}$ in $\mathbb C^n$,
defined as the set of all elements of the form
\begin{equation}
\label{eq:PartDiag}
\mathbf x=(\underbrace{x_1,\ldots,x_1}_{n_1\text{ times}},
\ldots,\underbrace{x_r,\ldots,x_r}_{n_r\text{ times}})
\end{equation}
for $(x_1,\ldots,x_r)\in\mathbb C^r$. The small diagonal is the particular
case $\mathbf n=(n)$; we denote it simply by $\Delta$. We define
$\BDConv_n\bigl|_{\Delta_{\mathbf n}}$ to be $\pi^{-1}(\Delta_{\mathbf n})$.

Given $g\in G^\vee\bigl(\mathbb C\bigl[z,z^{-1}\bigr]\bigr)$ and
$x\in\mathbb C$, we denote by $g_{|x}$ the result of substituting
$z-x$ for $z$ in~$g$. For $\mu\in\Lambda$ and $x\in\mathbb C$, we define
$$\widetilde T_{\mu|x}=
N^{-,\vee}\bigl(\mathbb C\bigl[z,(z-x)^{-1}\bigr]\bigr)\,(z-x)^\mu;$$
this is the set of all elements of the form $g_{|x}$ with
$g\in\widetilde T_\mu$. Given $\bm\mu=(\mu_1,\ldots,\mu_n)$ in
$\Lambda^n$, we define $\mathcal T_{\bm\mu}$ to be the set of all pairs
$(x_1,\ldots,x_n,[g_1,\ldots,g_n])$ with $(x_1,\ldots,x_n)\in\mathbb C^n$
and $g_j\in\widetilde T_{\mu_j|x_j}$ for each $j\in\{1,\ldots,n\}$.
For $\mu\in\Lambda$, we set (leaving $n$ out of the notation)
$$\dot T_\mu=\bigcup_{\substack{\bm\mu\in\Lambda^n\\[2pt]
|\bm\mu|=\mu}}\mathcal T_{\bm\mu}.$$

Given a $N^{-,\vee}(\mathbb C[z])$-invariant subset $Z\subset T_\mu$, we
denote by $\widetilde Z_{|x}$ the set of all elements of the form
$g_{|x}$ with $g\in\widetilde Z$. Given $(\mu_1,\ldots,\mu_n)\in\Lambda^n$
and $\mathbf Z=(Z_1,\ldots,Z_n)$ in $\mathscr Z(\lambda_1)_{\mu_1}\times
\cdots\times\mathscr Z(\lambda_n)_{\mu_n}$, we define
$\dot{\mathcal X}(\mathbf Z)$ to be the set of all pairs
$(x_1,\ldots,x_n;[g_1,\ldots,g_n])$ with $(x_1,\ldots,x_n)\in\mathbb C^n$
and $g_j\in\widetilde Z_{j|x_j}$ for each $j\in\{1,\ldots,n\}$.
(This subset $\dot{\mathcal X}(\mathbf Z)$ is denoted by
$\Uppsi(Z_1\caltimes\cdots\caltimes Z_n)$
in~\cite{BaumannGaussentLittelmann}.)
Given in addition a composition $\mathbf n$ of $n$, we define
$$\mathcal X(\mathbf Z,\mathbf n)=\overline{\dot{\mathcal X}(\mathbf Z)
\bigl|_{\Delta_{\mathbf n}}}\cap\BDConv_n^{\bm\lambda}$$
where the bar means closure in $\dot T_\mu\bigl|_{\Delta_{\mathbf n}}$.

For given $\bm\lambda$, $\mu$ and $\mathbf n$,
the subsets $\mathcal X(\mathbf Z,\mathbf n)$ for $\mathbf Z$ in
$$\mathscr Z(\bm\lambda)_\mu=
\bigsqcup_{\substack{(\mu_1,\ldots,\mu_n)\in\Lambda^n\\[2pt]
\mu_1+\cdots+\mu_n=\mu}}\mathscr Z(\lambda_1)_{\mu_1}
\times\cdots\times\,\mathscr Z(\lambda_n)_{\mu_n}$$
are the irreducible components of $\bigl(\BDConv_n^{\bm\lambda}\cap
\dot T_\mu\bigr)\bigl|_{\Delta_{\mathbf n}}$ (see
\cite{BaumannGaussentLittelmann}, proof of Proposition~5.4).
We adopt a special notation for the small diagonal and set
$\mathcal Y(\mathbf Z)=\mathcal X(\mathbf Z,(n))$.

Now fix $n$, the tuple $\bm\lambda\in(\Lambda^+)^n$, the weight
$\mu\in\Lambda$, and the composition $\mathbf n$ of $n$. We write
$\bm\lambda$ as a concatenation
$\bigl(\bm\lambda_{(1)},\ldots,\bm\lambda_{(r)}\bigr)$, where each
$\bm\lambda_{(j)}$ belongs to $(\Lambda^+)^{n_j}$, and similarly we
write each tuple $\mathbf Z\in\mathscr Z(\bm\lambda)_\mu$ as
$\strut\bigl(\mathbf Z_{(1)},\ldots,\mathbf Z_{(r)}\bigr)$ with
$\mathbf Z_{(j)}\in\mathscr Z(\bm\lambda_{(j)})$. Then
$$V(\bm\lambda)=V\bigl(\bm\lambda_{(1)}\bigr)\otimes\cdots\otimes
V\bigl(\bm\lambda_{(r)}\bigr)\quad\text{and}\quad
\bigl\langle\mathbf Z_{(j)}\bigr\rangle\in V\bigl(\bm\lambda_{(j)}\bigr).$$
With this notation (\cite{BaumannGaussentLittelmann}, Proposition~5.10):
\begin{proposition}
\label{pr:TransMat}
Let $(\mathbf Z',\mathbf Z'')\in(\mathscr Z(\bm\lambda)_\mu)^2$.
The coefficient $b_{\mathbf Z',\mathbf Z''}$ in the expansion
$$\bigl\langle\mathbf Z''_{(1)}\bigr\rangle\otimes\cdots\otimes
\bigl\langle\mathbf Z''_{(r)}\bigr\rangle=\sum_{\mathbf Z\in\mathscr Z
(\bm\lambda)_\mu}b_{\mathbf Z,\mathbf Z''}\;\langle\mathbf Z\rangle$$
is the multiplicity of $\mathcal Y(\mathbf Z')$ in the intersection
product $\mathcal X(\mathbf Z'',\mathbf n)\,\cdot\,
\BDConv_n^{\bm\lambda}\bigl|_\Delta$ computed in the
ambient~space $\BDConv_n^{\bm\lambda}\bigl|_{\Delta_{\mathbf n}}$.
\end{proposition}

\subsection{Projecting onto Cartan components}
\label{ss:ProjCart}
We continue with the setup of the previous section. First let $n$ be
a positive integer, let $\bm\lambda\in(\Lambda^+)^n$, and denote by
$p:V(\bm\lambda)\to V(|\bm\lambda|)$ the projection onto the Cartan
component, with kernel the sum of the other isotypical components of
$V(\bm\lambda)$.

The map $m_n:\Gr_n\to\Gr$ restricts to an isomorphism
$\Gr_n^{\bm\lambda}\cap(m_n)^{-1}\bigl(\Gr^{|\bm\lambda|}\bigr)\to
\Gr^{|\bm\lambda|}$ (see \cite{Haines}, p.~2110). Given $\mu\in\Lambda$
and $\mathbf Z\in\mathscr Z(\bm\lambda)_\mu$, we define $|\mathbf Z|$
to be the closure in $T_\mu$ of $m_n(\mathbf Z)\cap\Gr^{|\bm\lambda|}$.

The following proposition is a direct consequence of Theorem~3.3 in
\cite{BaumannGaussentLittelmann} and its proof.

\begin{proposition}
\label{pr:CompatIsot}
\begin{enumerate}
\item
The map $\mathbf Z\mapsto|\mathbf Z|$ defines a bijection
$\bigl\{\mathbf Z\in\mathscr Z(\bm\lambda)\bigm||
\mathbf Z|\neq\varnothing\bigr\}\to\mathscr Z(|\bm\lambda|)$.
\item
Let $\mathbf Z\in\mathscr Z(\bm\lambda)$. If $|\mathbf Z|\neq\varnothing$,
then $p(\langle\mathbf Z\rangle)=\bigl\langle|\mathbf Z|\bigr\rangle$;
otherwise $p(\langle\mathbf Z\rangle)=0$.
\end{enumerate}
\end{proposition}

Now let $\mathbf n=(n_1,\ldots,n_r)$ be a composition of $n$ in $r$ parts.
We again write $\bm\lambda$ as a concatenation $\bigl(\bm\lambda_{(1)},
\ldots,\bm\lambda_{(r)}\bigr)$, where each $\bm\lambda_{(j)}$ belongs
to $(\Lambda^+)^{n_j}$, and set $\|\bm\lambda\|=\bigl(\bigl|
\bm\lambda_{(1)}\bigr|,\ldots,\bigl|\bm\lambda_{(r)}\bigr|\bigr)$; then
$$V(\|\bm\lambda\|)=V\bigl(\bigl|\bm\lambda_{(1)}\bigr|\bigr)
\otimes\cdots\otimes V\bigl(\bigl|\bm\lambda_{(r)}\bigr|\bigr).$$
For each $j\in\{1,\ldots,r\}$, we denote by $p_{(j)}:V\bigl(
\bm\lambda_{(j)}\bigr)\to V\bigl(\bigl|\bm\lambda_{(j)}\bigr|\bigr)$
the projection onto the Cartan component and define
$$\mathbf p=p_{(1)}\otimes\cdots\otimes p_{(r)};$$
thus $\mathbf p:V(\bm\lambda)\to V(\|\bm\lambda\|)$.

Likewise, we again write each tuple $\mathbf Z\in\mathscr Z(\bm\lambda)$
as a concatenation $\bigl(\mathbf Z_{(1)},\ldots,\mathbf Z_{(r)}\bigr)$
with $\mathbf Z_{(j)}\in\mathscr Z(\bm\lambda_{(j)})$ and set
$\|\mathbf Z\|=\bigl(\bigl|\mathbf Z_{(1)}\bigr|,\ldots,
\bigl|\mathbf Z_{(r)}\bigr|\bigr)$.

\begin{proposition}
\label{pr:ProjCart}
Let $\mathbf Z\in\mathscr Z(\bm\lambda)$.
If $\bigl|\mathbf Z_{(j)}\bigr|\neq\varnothing$
for all $j\in\{1,\ldots,r\}$, then
$\mathbf p(\langle\mathbf Z\rangle)=\bigl\langle\|\mathbf Z\|\bigr\rangle$;
otherwise $p(\langle\mathbf Z\rangle)=0$.
\end{proposition}
\begin{proof}
Let $\mathring{\mathscr Z}(\bm\lambda)$ be the set of all
$\mathbf Z\in\mathscr Z(\bm\lambda)$ such that
$\bigl|\mathbf Z_{(j)}\bigr|\neq\varnothing$ for all $j\in\{1,\ldots,r\}$;
then the map $\mathbf Z\mapsto\|\mathbf Z\|$ realizes a bijection from
$\mathring{\mathscr Z}(\bm\lambda)$ onto $\mathscr Z(\|\bm\lambda\|)$.

We fix a weight $\mu\in\Lambda$ and introduce the transition matrices
$(b_{\mathbf Z',\mathbf Z''})$ and $(a_{\mathbf Y',\mathbf Y''})$,
where $(\mathbf Z',\mathbf Z'')\in(\mathscr Z(\bm\lambda)_\mu)^2$
and $(\mathbf Y',\mathbf Y'')\in(\mathscr Z(\|\bm\lambda\|)_\mu)^2$,
that encode the expansions
$$\bigl\langle\mathbf Z''_{(1)}\bigr\rangle\otimes\cdots\otimes
\bigl\langle\mathbf Z''_{(r)}\bigr\rangle=\sum_{\mathbf Z'\in\mathscr Z
(\bm\lambda)_\mu}b_{\mathbf Z',\mathbf Z''}\;\langle\mathbf Z'\rangle$$
and
$$\bigl\langle Y''_1\bigr\rangle\otimes\cdots\otimes
\bigl\langle Y''_r\bigr\rangle=\sum_{\mathbf Y'\in\mathscr Z
(\|\bm\lambda\|)_\mu}a_{\mathbf Y',\mathbf Y''}\;\langle\mathbf Y'\rangle$$
on the MV bases of $V(\bm\lambda)$ and $V(\|\bm\lambda\|)$.
We claim that if $\mathbf Z'\in\mathring{\mathscr Z}(\bm\lambda)$, then
\begin{equation}
\label{eq:ClaimTM}
b_{\mathbf Z',\mathbf Z''}=
\begin{cases}
a_{\|\mathbf Z'\|,\|\mathbf Z''\|}&
\text{if $\mathbf Z''\in\mathring{\mathscr Z}(\bm\lambda)$,}\\[2pt]
0&\text{otherwise.}
\end{cases}
\end{equation}
Assuming \eqref{eq:ClaimTM}, we conclude the proof as follows. Let
$\widetilde{\mathbf p}:V(\bm\lambda)\to V(\|\bm\lambda\|)$ be the linear
map defined by the requirement that for all
$\mathbf Z\in\mathscr Z(\bm\lambda)$,
$$\widetilde{\mathbf p}(\langle\mathbf Z\rangle)=
\begin{cases}
\bigl\langle\|\mathbf Z\|\bigr\rangle&
\text{if $\mathbf Z\in\mathring{\mathscr Z}(\bm\lambda)$,}\\[2pt]
0&\text{otherwise.}
\end{cases}$$
Then \eqref{eq:ClaimTM} gives
$$\widetilde{\mathbf p}\bigl(\bigl\langle\mathbf Z_{(1)}\bigr\rangle
\otimes\cdots\otimes\bigl\langle\mathbf Z_{(r)}\bigr\rangle\bigr)=
\begin{cases}
\bigl\langle|\mathbf Z_{(1)}|\bigr\rangle\otimes\cdots\otimes
\bigl\langle|\mathbf Z_{(r)}|\bigr\rangle&
\text{if $\mathbf Z\in\mathring{\mathscr Z}(\bm\lambda)$,}\\[2pt]
0&\text{otherwise,}
\end{cases}$$
and from Proposition~\ref{pr:CompatIsot}, we conclude that
$\widetilde{\mathbf p}=\mathbf p$.

We are thus reduced to prove~\eqref{eq:ClaimTM}. We define a map
$\mathbf{m_n}:\BDConv_n\bigl|_{\Delta_{\mathbf n}}\to\BDConv_r$ by
$$\mathbf{m_n}(\mathbf x;[g_1,\ldots,g_n])=
(x_1,\ldots,x_r;[g_1\cdots g_{n_1},\ g_{n_1+1}\cdots
g_{n_1+n_2},\ \ldots,\ g_{n_1+\ldots+n_{r-1}+1}\cdots g_n])$$
for $\mathbf x$ as in \eqref{eq:PartDiag}. Then
$\mathcal U=\BDConv_n^{\bm\lambda}\bigl|_{\Delta_{\mathbf n}}\cap
(\mathbf{m_n})^{-1}\Bigl(\BDConv_r^{\|\bm\lambda\|}\Bigr)$ is an open
subset of $\BDConv_n^{\bm\lambda}\bigl|_{\Delta_{\mathbf n}}$ and
$\mathbf{m_n}$ restricts to an isomorphism
$\mathcal U\to\BDConv_r^{\|\bm\lambda\|}$.

Let $(\mathbf Z',\mathbf Z'')\in(\mathscr Z(\bm\lambda)_\mu)^2$.
By Proposition~\ref{pr:TransMat}, the
coefficient $b_{\mathbf Z',\mathbf Z''}$ is the multiplicity of
$\mathcal Y(\mathbf Z')$ in the intersection product
$\mathcal X(\mathbf Z'',\mathbf n)\,\cdot
\bigl(\BDConv_n^{\bm\lambda}\bigr)\bigl|_\Delta$ computed in the
ambient space $\BDConv_n^{\bm\lambda}\bigl|_{\Delta_{\mathbf n}}$.

Assume first that both $\mathbf Z'$ and $\mathbf Z''$ lie in
$\mathring{\mathscr Z}(\bm\lambda)$. Then the open subset $\mathcal U$
meets $\mathcal Y(\mathbf Z')$ and $\mathcal X(\mathbf Z'',\mathbf n)$.
Since intersection multiplicities are of local nature,
$b_{\mathbf Z',\mathbf Z''}$ is the multiplicity of
$\mathcal Y(\mathbf Z')\cap\mathcal U$ in the intersection product
$\bigl(\mathcal X(\mathbf Z'',\mathbf n)\cap\mathcal U\bigr)\,\cdot
\mathcal U\bigl|_\Delta$ computed in the ambient space
$\mathcal U\bigl|_{\Delta_{\mathbf n}}$. On the other hand,
Proposition \ref{pr:TransMat} for the composition $(1^r)=(1,\ldots,1)$
of $r$ gives that $a_{\|\mathbf Z'\|,\|\mathbf Z''\|}$ is
the multiplicity of $\mathcal Y(\|\mathbf Z'\|)$ in the intersection
product $\mathcal X(\|\mathbf Z''\|,(1^r))\,\cdot
\bigl(\BDConv_r^{\|\bm\lambda\|}\bigr)\bigl|_\Delta$ computed in the
ambient space $\BDConv_r^{\|\bm\lambda\|}$.
Observing that
$$\mathbf{m_n}\bigl(\mathcal Y(\mathbf Z')\cap\mathcal U\bigr)
=\mathcal Y(\|\mathbf Z'\|)\quad\text{and}\quad
\mathbf{m_n}\bigl(\mathcal X(\mathbf Z'',\mathbf n)\cap
\mathcal U\bigr)=\mathcal X(\|\mathbf Z''\|,(1^r)),$$
we conclude that
$b_{\mathbf Z',\mathbf Z''}=a_{\|\mathbf Z'\|,\|\mathbf Z''\|}$
in this case.

Now assume that $\mathbf Z'$ is in $\mathring{\mathscr Z}(\bm\lambda)$
but not $\mathbf Z''$. Then there exists $j\in\{1,\ldots,r\}$ such that
$\mathbf Z''_{(j)}$ is contained in
$F=\overline{\Gr_{n_j}^{\bm\lambda_{(j)}}}
\setminus(m_{n_j})^{-1}\bigl(\Gr^{|\bm\lambda_{(j)}|}\bigr)$.
For $x\in\mathbb C$, denote by $\widehat F_{\,|x}$ the set of all
tuples $\bigl(g_{1|x},\ldots,g_{n_j|x}\bigr)$ where
$$(g_1,\ldots,g_{n_j})\in\bigl(G^\vee\bigl(\mathbb C\bigl[z,z^{-1}\bigr]
\bigr)\bigr)^{n_j}\quad\text{and}\quad[g_1,\ldots,g_{n_j}]\in F;$$
denote by $\mathcal F$ the subset of
$\BDConv_n^{\bm\lambda}\bigl|_{\Delta_{\mathbf n}}$ consisting
of all pairs $(\mathbf x;[g_1,\ldots,g_n])$ such that
$$(g_{n_1+\cdots+n_{j-1}+1},\ldots,g_{n_1+\cdots+n_j})\in
\widehat F_{\,|x_j}$$
where $\mathbf x$ is written as in \eqref{eq:PartDiag}.
Then $F$ is closed in $\overline{\Gr_{n_j}^{\bm\lambda_{(j)}}}$ and
$\mathcal X(\mathbf Z'',\mathbf n)$ is contained in $\mathcal F$.
As $\mathcal Y(\mathbf Z')$ is not contained in $\mathcal F$,
it is not contained in $\mathcal X(\mathbf Z'',\mathbf n)$,
so here $b_{\mathbf Z',\mathbf Z''}=0$.
\end{proof}

\subsection{Truncation}
\label{ss:Trunc}
In this section, we come back to the setup of sect.~\ref{ss:TransMat}
and record a property which will simplify our analysis.

We fix nonnegative integers $n_1$, $n_2$, $n_3$ and tuples
$\bm\lambda_{(1)}\in(\Lambda^+)^{n_1}$,
$\bm\lambda_{(2)}\in(\Lambda^+)^{n_2}$,
$\bm\lambda_{(3)}\in(\Lambda^+)^{n_3}$.
We define $\bm\lambda$ to be the concatenation
$\bigl(\bm\lambda_{(1)},\bm\lambda_{(2)},\bm\lambda_{(3)}\bigr)$ and we
regard elements $\mathbf Z\in\mathscr Z(\bm\lambda)$ as concatenations
$\bigl(\mathbf Z_{(1)},\mathbf Z_{(2)},\mathbf Z_{(3)}\bigr)$
where each $\mathbf Z_{(j)}$ belongs to $\mathscr Z(\bm\lambda_{(j)})$.
If $\nu\in\Lambda$ and $\mathbf Z\in\mathscr Z(\bm\lambda_{(3)})_\nu$,
then we set $\wt\mathbf Z=\nu$.

We fix a weight $\mu\in\Lambda$ and introduce the transition
matrix $(a_{\mathbf Z',\mathbf Z''})$,
where $(\mathbf Z',\mathbf Z'')\in(\mathscr Z(\bm\lambda)_\mu)^2$,
that encodes the expansions
$$\bigl\langle\mathbf Z''_{(1)}\bigr\rangle\otimes\bigl\langle
\bigl(\mathbf Z''_{(2)},\mathbf Z''_{(3)}\bigr)\bigr\rangle=
\sum_{\mathbf Z'\in\mathscr Z(\bm\lambda)_\mu}
a_{\mathbf Z',\mathbf Z''}\;\bigl\langle\bigl(\mathbf Z'_{(1)},
\mathbf Z'_{(2)},\mathbf Z'_{(3)}\bigr)\bigr\rangle$$
on the MV basis of $V(\bm\lambda)$.

\begin{proposition}
\label{pr:Trunc}
\begin{enumerate}
\item
Let $(\mathbf Z',\mathbf Z'')\in(\mathscr Z(\bm\lambda)_\mu)^2$.
If $a_{\mathbf Z',\mathbf Z''}\neq0$, then either $\wt\mathbf Z'_{(3)}
<\wt\mathbf Z''_{(3)}$ or $\mathbf Z'_{(3)}=\mathbf Z''_{(3)}$.
\item
Let $\mathbf Z''\in\mathscr Z(\bm\lambda)_\mu$. Then
$$\bigl\langle\mathbf Z''_{(1)}\bigr\rangle\otimes
\bigl\langle\mathbf Z''_{(2)}\bigr\rangle=
\sum_{\substack{\mathbf Z'\in\mathscr Z(\bm\lambda)_\mu\\[2pt]
\mathbf Z'_{(3)}=\mathbf Z''_{(3)}}}
a_{\mathbf Z',\mathbf Z''}\;\bigl\langle\bigl(\mathbf Z'_{(1)},
\mathbf Z'_{(2)}\bigr)\bigr\rangle$$
in $V\bigl(\bm\lambda_{(1)}\bigr)\otimes V\bigl(\bm\lambda_{(2)}\bigr)$.
\end{enumerate}
\end{proposition}
\begin{proof}
Let $\mathbf Z''\in\mathscr Z(\bm\lambda)_\mu$ and set $\nu=\wt Z''_{(3)}$.
We can expand
$$\bigl\langle\mathbf Z''_{(1)}\bigr\rangle\otimes
\bigl\langle\mathbf Z''_{(2)}\bigr\rangle=
\sum_{\mathbf Z\in\mathscr Z(\bm\lambda_{(1)},\bm\lambda_{(2)})_{\mu-\nu}}
c_{\mathbf Z,\mathbf Z''}\;\langle\mathbf Z\rangle$$
on the MV basis of
$V\bigl(\bm\lambda_{(1)}\bigr)\otimes V\bigl(\bm\lambda_{(2)}\bigr)$.

We denote by $V\bigl(\bm\lambda_{(3)}\bigr){}^{}_{<\nu}$ the sum of the
$\xi$-weight subspaces of $V\bigl(\bm\lambda_{(3)}\bigr)$ with
$\xi<\nu$. By~Theorem~5.13
in~\cite{BaumannGaussentLittelmann},
$$\bigl\langle\mathbf Z''_{(2)}\bigr\rangle\otimes
\bigl\langle\mathbf Z''_{(3)}\bigr\rangle\equiv
\bigl\langle\bigl(\mathbf Z''_{(2)},\mathbf Z''_{(3)}\bigr)\bigr\rangle
\quad\bigl(\bmod\ V\bigl(\bm\lambda_{(2)}\bigr)\otimes
V\bigl(\bm\lambda_{(3)}\bigr){}^{}_{<\nu}\bigr)$$
and for each $\mathbf Z\in\mathscr Z(\bm\lambda_{(1)},\bm\lambda_{(2)})$,
$$\bigl\langle\mathbf Z\bigr\rangle\otimes
\bigl\langle\mathbf Z''_{(3)}\bigr\rangle\equiv
\bigl\langle\bigl(\mathbf Z,\mathbf Z''_{(3)}\bigr)\bigr\rangle
\quad\bigl(\bmod\ V\bigl(\bm\lambda_{(1)}\bigr)\otimes
V\bigl(\bm\lambda_{(2)}\bigr)\otimes
V\bigl(\bm\lambda_{(3)}\bigr){}^{}_{<\nu}\bigr).$$
Thus,
$$\sum_{\mathbf Z'\in\mathscr Z(\bm\lambda)_\mu}
a_{\mathbf Z',\mathbf Z''}\;\bigl\langle\bigl(\mathbf Z'_{(1)},
\mathbf Z'_{(2)},\mathbf Z'_{(3)}\bigr)\bigr\rangle\equiv
\sum_{\mathbf Z\in\mathscr Z(\bm\lambda_{(1)},\bm\lambda_{(2)})_{\mu-\nu}}
c_{\mathbf Z,\mathbf Z''}\;
\bigl\langle\bigl(\mathbf Z,\mathbf Z''_{(3)}\bigr)\bigr\rangle$$
modulo $V\bigl(\bm\lambda_{(1)}\bigr)\otimes
V\bigl(\bm\lambda_{(2)}\bigr)\otimes
V\bigl(\bm\lambda_{(3)}\bigr){}^{}_{<\nu}$. We conclude by noting, by
means of Proposition~5.11 in~\cite{BaumannGaussentLittelmann}, that the
latter space is spanned by the basis vectors $\langle\mathbf Z'\rangle$
such that $\wt\mathbf Z'_{(3)}<\nu$.
\end{proof}

\section{Geometry}
\label{se:Geometry}
In this section, we prove that the MV basis of the tensor powers of the
natural representation of $G=\SL_2(\mathbb C)$ is the basis $(y_w)$ from
sect.~\ref{se:CombLin}. As a matter of fact, by Theorem~5.13
in~\cite{BaumannGaussentLittelmann}, the MV basis satisfies the first
equation in \eqref{eq:DefYw}, so we only have to prove that it satisfies
the second one too.

\subsection{Notation}
We endow $G$ with its usual maximal torus and Borel subgroup. The weight
lattice is represented as usual as the quotient $(\mathbb Z\varepsilon_1
\oplus\mathbb Z\varepsilon_2)/\mathbb Z(\varepsilon_1+\varepsilon_2)$.
The fundamental weight $\varpi$ is the image of $\varepsilon_1$ in this
quotient. The notation $\Gr$ indicates the affine Grassmannian of
$G^\vee=\PGL_2(\mathbb C)$.

In this section, $\bm\lambda$ will always be of the form
$(\varpi,\ldots,\varpi)$; the number $n$ of times $\varpi$ is repeated will
usually appears as a subscript in notation like $\Gr_n^{\bm\lambda}$ or
$\BDConv_n^{\bm\lambda}$.

The cell $\Gr^{\varpi}$ is isomorphic to the projective line, hence is
closed. The two MV cycles in $\mathscr Z(\varpi)$ are
$$Z_+=\Gr^{\varpi}\cap T_{\varpi}=\biggl\{\biggl[\begin{pmatrix}z&0\\
0&1\end{pmatrix}\biggr]\biggr\}\quad\text{and}\quad
Z_-=\Gr^{\varpi}\cap T_{-\varpi}=\biggl\{\biggl[\begin{pmatrix}1&0\\
a&z\end{pmatrix}\biggr]\biggm|a\in\mathbb C\biggr\}$$
(the matrices above should actually be viewed in
$\PGL_2\bigl(\mathbb C\bigl[z,z^{-1}\bigr]\bigr)$).
The standard basis of $V(\varpi)=\mathbb C^2$ is then
$(x_+,x_-)=(\langle Z_+\rangle,\langle Z_-\rangle)$.

Given a word $v\in\mathscr C_n$, we set
$$P(v)=\bigl\{\ell\in\{1,\ldots,n\}\bigm|v(\ell)=+\bigr\}\quad\text{and}
\quad\mathbf Z_v=\bigl(Z_{v(1)},\ldots,Z_{v(n)}\bigr).$$
Thanks to the bijection~\eqref{eq:ProdCycMV}, we regard $\mathbf Z_v$
as an element in $\mathscr Z(\bm\lambda)$.

For $(x,a)\in\mathbb C^2$, we set
$$\varphi_+(x,a)=\begin{pmatrix}z-x&a\\0&1\end{pmatrix}\quad\text{and}\quad
\varphi_-(x,a)=\begin{pmatrix}1&0\\a&z-x\end{pmatrix}.$$
Recall the notation introduced in sect.~\ref{ss:TransMat}.
For each word $v\in\mathscr C_n$, we define an embedding
$\upphi_v:\mathbb C^{2n}\to\BDConv_n^{\bm\lambda}$ by
$$\upphi_v(\mathbf x;\mathbf a)=\bigl(\mathbf x;\bigl[\varphi_{v(1)}
(x_1,a_1),\ldots,\varphi_{v(n)}(x_n,a_n)\bigr]\bigr)$$
where $\mathbf x=(x_1,\ldots,x_n)$ and $\mathbf a=(a_1,\ldots,a_n)$.
The image of $\upphi_v$ is an open subset $U_v$ and $\upphi_v$ can be
regarded as a chart on the manifold $\BDConv_n^{\bm\lambda}$. This
chart is designed so that $\dot{\mathcal X}(\mathbf Z_v)$ is the
algebraic subset of $U_v$ defined by the equations $a_\ell=0$ for
$\ell\in P(v)$ (compare with the construction presented
in~\cite{GaussentLittelmann}).

\subsection{The simplest example}
In this section, we consider the case $n=2$; the variety
$\BDConv_2^{\bm\lambda}$ has dimension $4$. The words $v=+-$ and
$w=-+$ give rise to charts $\upphi_v$ and $\upphi_w$ on
$\BDConv_2^{\bm\lambda}$ defined by
\begin{align*}
\upphi_v(x_1,x_2;a_1,a_2)&=\biggl(x_1,x_2;\biggl[
\begin{pmatrix}z-x_1&a_1\\0&1\end{pmatrix},
\begin{pmatrix}1&0\\a_2&z-x_2\end{pmatrix}\biggr]\biggr),\\[6pt]
\upphi_w(x_1,x_2;b_1,b_2)&=\biggl(x_1,x_2;\biggl[
\begin{pmatrix}1&0\\b_1&z-x_1\end{pmatrix},
\begin{pmatrix}z-x_2&b_2\\0&1\end{pmatrix}\biggr]\biggr).
\end{align*}

The transition map $(\upphi_w)^{-1}\circ\upphi_v$ is given by
$$b_1=1/a_1,\quad\;b_2=-a_1(x_2-x_1+a_1a_2)$$
on the domain
$$(\upphi_v)^{-1}(U_v\cap U_w)=\bigl\{(x_1,x_2,a_1,a_2)\in\mathbb C^4
\bigm|a_1\neq0\bigr\}.$$
We set $A=\mathbb C[x_1,x_2,a_1,a_2]$; this is the coordinate ring of
$(\upphi_v)^{-1}(U_v)$. We let $B=\mathscr S^{-1}A$ be the localization
of $A$ with respect to the multiplicative subset $\mathscr S$ generated
by~$a_1$; this is the coordinate ring of $(\upphi_v)^{-1}(U_v\cap U_w)$.

In the chart $\upphi_v$, the cycle $\mathcal Y(\mathbf Z_v)$ is defined
by the equations $a_1=x_1-x_2=0$, so the ideal in $A$ of the subvariety
$$V=(\upphi_v)^{-1}(\mathcal Y(\mathbf Z_v))$$
is
$$\mathfrak p=(a_1,x_1-x_2).$$
In the chart $\upphi_w$, the cycle $\dot{\mathcal X}(\mathbf Z_w)$
is defined by the equation $b_2=0$, and the closure in $U_v$
of $U_v\cap\dot{\mathcal X}(\mathbf Z_w)$ is
$U_v\cap\mathcal X(\mathbf Z_w,(1,1))$. Therefore the ideal in
$B$ of $(\upphi_v)^{-1}\bigl(U_v\cap\dot{\mathcal X}(\mathbf Z_w)\bigr)$
is $\mathring{\mathfrak q}=(-a_1(x_2-x_1+a_1a_2))$ and the ideal
in $A$ of the subvariety
$$X=(\upphi_v)^{-1}(U_v\cap\mathcal X(\mathbf Z_w,(1,1)))$$
is the preimage
$$\mathfrak q=(x_2-x_1+a_1a_2)$$
of $\mathring{\mathfrak q}$ under the canonical map $A\to B$.

Plainly $\mathfrak q\subset\mathfrak p$, which shows that $V\subset X$.
The local ring $\mathscr O_{V,X}$ of $X$ along $V$ is the localization
of $\overline A=A/\mathfrak q$ at the ideal
$\overline{\mathfrak p}=\mathfrak p/\mathfrak q$. Since $a_2$ is not
in $\mathfrak p$, its image in $\overline A_{\,\overline{\mathfrak p}}$
is invertible, and then we see that $x_1-x_2$ generates the maximal ideal
of $\overline A_{\,\overline{\mathfrak p}}$. As a consequence, the order
of vanishing of $x_1-x_2$ along $V$ (see~\cite{Fulton}, sect.~1.2) is
equal to one. By definition, this is the multiplicity of
$\mathcal Y(\mathbf Z_v)$ in the intersection product
$\mathcal X(\mathbf Z_w,(1,1))\cdot\,\BDConv_2^{\bm\lambda}\bigl|_\Delta$.

Proposition~\ref{pr:TransMat} then asserts that
$y_{+-}=\langle\mathbf Z_v\rangle$ occurs with coefficient one
in the expansion of
$x_w=\langle\mathbf Z_-\rangle\otimes\langle\mathbf Z_+\rangle$ on
the MV basis of $V(\varpi)^{\otimes2}$, in agreement with the equation
$$x_{-+}=y_{-+}+y_{+-}.$$
The proof of the general case follows the same pattern, but more
elaborate combinatorics is needed to manage the equations.

\subsection{Transition maps}
\label{ss:TransMaps}
Pick $v$, $w$ in $\mathscr C_n$. Set $P_0=S_0=1$ and $Q_0=R_0=0$.
For $\ell\in\{1,\ldots,n\}$, let $K_\ell=\mathbb C(x_1,\ldots,x_\ell,
a_1,\ldots,a_\ell)$ be the field of rational functions and define by
induction an element $b_\ell\in K_\ell$ and a matrix
$$\begin{pmatrix}P_\ell&Q_\ell\\R_\ell&S_\ell\end{pmatrix}$$
with coefficients in $K_\ell[z]$ and determinant one as follows:
\begin{itemize}
\item
If $(v(\ell),w(\ell))=(+,+)$, then
$$b_\ell=\frac{\bigl(a_\ell P_{\ell-1}+Q_{\ell-1}\bigr)\bigl(x_\ell\bigr)}
{\bigl(a_\ell R_{\ell-1}+S_{\ell-1}\bigr)\bigl(x_\ell\bigr)},\qquad
\left\{\begin{alignedat}2
P_\ell&=P_{\ell-1}-b_\ell R_{\ell-1},\quad\;&Q_\ell&=\frac{a_\ell
P_{\ell-1}+Q_{\ell-1}-b_\ell S_\ell}{z-x_\ell},\\
R_\ell&=(z-x_\ell)R_{\ell-1},\quad\;&S_\ell&=a_\ell R_{\ell-1}+S_{\ell-1}.
\end{alignedat}\right.$$
\item
If $(v(\ell),w(\ell))=(-,+)$, then
$$b_\ell=\frac{\bigl(P_{\ell-1}+a_\ell Q_{\ell-1}\bigr)\bigl(x_\ell\bigr)}
{\bigl(R_{\ell-1}+a_\ell S_{\ell-1}\bigr)\bigl(x_\ell\bigr)},\qquad
\left\{\begin{alignedat}2
P_\ell&=\frac{P_{\ell-1}+a_\ell Q_{\ell-1}-b_\ell R_\ell}{z-x_\ell},
\quad\;&Q_\ell&=Q_{\ell-1}-b_\ell S_{\ell-1},\\
R_\ell&=R_{\ell-1}+a_\ell S_{\ell-1},\quad\;&S_\ell&=(z-x_\ell)S_{\ell-1}.
\end{alignedat}\right.$$
\item
If $(v(\ell),w(\ell))=(+,-)$, then
$$b_\ell=\frac{\bigl(a_\ell R_{\ell-1}+S_{\ell-1}\bigr)\bigl(x_\ell\bigr)}
{\bigl(a_\ell P_{\ell-1}+Q_{\ell-1}\bigr)\bigl(x_\ell\bigr)},\qquad
\left\{\begin{alignedat}2
P_\ell&=(z-x_\ell)P_{\ell-1},\quad\;&Q_\ell&=a_\ell P_{\ell-1}+Q_{\ell-1},\\
R_\ell&=R_{\ell-1}-b_\ell P_{\ell-1},\quad\;&S_\ell&=\frac{a_\ell
R_{\ell-1}+S_{\ell-1}-b_\ell Q_\ell}{z-x_\ell}.
\end{alignedat}\right.$$
\item
If $(v(\ell),w(\ell))=(-,-)$, then
$$b_\ell=\frac{\bigl(R_{\ell-1}+a_\ell S_{\ell-1}\bigr)\bigl(x_\ell\bigr)}
{\bigl(P_{\ell-1}+a_\ell Q_{\ell-1}\bigr)\bigl(x_\ell\bigr)},\qquad
\left\{\begin{alignedat}2
P_\ell&=P_{\ell-1}+a_\ell Q_{\ell-1},\quad\;&Q_\ell&=(z-x_\ell)Q_{\ell-1},\\
R_\ell&=\frac{R_{\ell-1}+a_\ell S_{\ell-1}-b_\ell P_\ell}{z-x_\ell},
\quad\;&S_\ell&=S_{\ell-1}-b_\ell Q_{\ell-1}.
\end{alignedat}\right.$$
\end{itemize}

Since the matrix $\begin{pmatrix}P_{\ell-1}&Q_{\ell-1}\\
R_{\ell-1}&S_{\ell-1}\end{pmatrix}$ has determinant one,
the denominator in the fraction that defines $b_\ell$ is not the
zero polynomial and everything is well-defined.

\begin{proposition}
The transition map
$$(\upphi_w)^{-1}\circ\upphi_v:\upphi_v^{-1}(U_w)\to\upphi_w^{-1}(U_v)$$
is given by the rational map
$$(x_1,\ldots,x_n;a_1,\ldots,a_n)\mapsto(x_1,\ldots,x_n;b_1,\ldots,b_n)$$
where $b_1$, \dots, $b_n$ are defined above.
\end{proposition}
\begin{proof}
The definitions are set up so that
$$\varphi_{w(\ell)}(x_\ell,b_\ell)
\begin{pmatrix}P_\ell&Q_\ell\\R_\ell&S_\ell\end{pmatrix}
=\begin{pmatrix}P_{\ell-1}&Q_{\ell-1}\\R_{\ell-1}&S_{\ell-1}\end{pmatrix}
\varphi_{v(\ell)}(x_\ell,a_\ell)$$
and therefore
$$\Biggl(\prod_{j=1}^\ell\varphi_{w(j)}(x_j,b_j)\Biggr)
\begin{pmatrix}P_\ell&Q_\ell\\R_\ell&S_\ell\end{pmatrix}
=\Biggl(\prod_{j=1}^\ell\varphi_{v(j)}(x_j,a_j)\Biggr)$$
for each $\ell\in\{1,\ldots,n\}$. Thus, when complex values are assigned
to the indeterminates $x_1$, \dots, $x_n$, $a_1$, \dots, $a_n$, we get
$$\Biggl[\prod_{j=1}^\ell\varphi_{v(j)}(x_j,a_j)\Biggr]
=\Biggl[\prod_{j=1}^\ell\varphi_{w(j)}(x_j,b_j)\Biggr]$$
in $\PGL_2\bigl(\mathbb C\bigl[z,(z-x_1)^{-1},\ldots,(z-x_\ell)^{-1}
\bigr]\bigr)\,/\,\PGL_2(\mathbb C[z])$. This implies the equality
$$\upphi_v(x_1,\ldots,x_n;a_1,\ldots,a_n)=
\upphi_w(x_1,\ldots,x_n;b_1,\ldots,b_n)$$
in $\BDConv_n$.
\end{proof}

The parameters $b_\ell$ and the coefficients of the polynomials
$P_\ell$, $Q_\ell$, $R_\ell$, $S_\ell$ were defined as elements in
$K_\ell$. We can however be more precise and define recursively a
subring $B_\ell\subset K_\ell$ to which they belong: we start with
$B_0=\mathbb C$, and for $\ell\in\{1,\ldots,n\}$, we set
$B_\ell=B_{\ell-1}\bigl[x_\ell,a_\ell,f_\ell^{-1}\bigr]$, where
$f_\ell\in B_{\ell-1}[x_\ell,a_\ell]$ is the denominator in the
fraction that defines~$b_\ell$.

Let $A_\ell=\mathbb C[x_1,\ldots,x_\ell,a_1,\ldots,a_\ell]$ be the
polynomial algebra. One can then easily build by induction a finitely
generated multiplicative set $\mathscr S_\ell\subset A_\ell$ such that
$B_\ell$ is the localization $\mathscr S_\ell^{-1}A_\ell$. While $A_n$
is the coordinate ring of $(\upphi_v)^{-1}(U_v)$, we see that $B_n$ is
the coordinate ring of the open subset $(\upphi_v)^{-1}(U_v\cap U_w)$.
In fact, since the matrix $\begin{pmatrix}P_\ell&Q_\ell\\R_\ell&S_\ell
\end{pmatrix}$ has determinant one, the numerator and the denominator
of $b_\ell$ cannot both vanish at the same time. As a consequence,
$(\upphi_w)^{-1}\circ\upphi_v$ cannot be defined at a point where a
function in $\mathscr S_n$ vanishes.

\subsection{Finding the equations}
\label{ss:FindEqns}
To prove that the MV basis satisfies the equation~\eqref{eq:DefYw},
we need intersection multiplicities in the ambient space
$\BDConv_n^{\bm\lambda}\bigl|_{\Delta_{(1,n-1)}}$. In practice,
we make the base change $\Delta_{(1,n-1)}\to\mathbb C^n$ by letting
$x_2=\cdots=x_n$ in the definition of the charts and by agreeing that
\textbf{from now on, $U_v$ actually means $U_v\bigl|_{\Delta_{(1,n-1)}}$.}
Then, in view of the invariance of the whole system under translation
along the small diagonal $\Delta$, all our equations will only involve
the difference $x=x_1-x_2$.

We will consider words $v$ and $w$ in $\mathscr C_n$ such that
$(v(1),w(1))=(+,-)$ and $\wt(v)=\wt(w)$. The planar paths that
represent $v$ and $w$ have then the same endpoints. We write
$w$ as a concatenation $-w'$ where $w'\in\mathscr C_{n-1}$.
Proposition~\ref{pr:TransMat} asserts that the basis element $y_v$
occurs in the expansion of $x_-\otimes y_{w'}$ on the MV basis of
$V(\varpi)^{\otimes n}$ only if
$\mathcal Y(\mathbf Z_v)\subset\mathcal X(\mathbf Z_w,(1,n-1))$,
and when this condition is fulfilled, its coefficient is the multiplicity
of $\mathcal Y(\mathbf Z_v)$ in the intersection product
$\mathcal X(\mathbf Z_w,(1,n-1))\cdot\,\BDConv_n^{\bm\lambda}
\bigl|_{\Delta}$.

The next sections are devoted to the determination of these inclusions and
intersection multiplicities. The actual calculations require the ideals in
$A_n$ of the subvarieties $(\upphi_v)^{-1}(\mathcal Y(\mathbf Z_v))$
and $(\upphi_v)^{-1}(U_v\cap\mathcal X(\mathbf Z_w,(1,n-1)))$ of
$(\upphi_v)^{-1}(U_v)$: the first one, denoted by $\mathfrak p$,
is generated by $x$ and the elements $a_\ell$ for $\ell\in P(v)$;
the second one, denoted by $\mathfrak q$, is less easily determined.

Taking into account our notational convention regarding the base change
$\Delta_{(1,n-1)}\to\mathbb C^n$, we observe that
$U_v\cap\mathcal X(\mathbf Z_w,(1,n-1))$ is the closure in $U_v$ of
$U_v\cap\dot{\mathcal X}(\mathbf Z_w)$. Now let $\mathring{\mathfrak q}_n$
be the ideal in $B_n$ of the closed subset
$(\upphi_v)^{-1}\bigl(U_v\cap\dot{\mathcal X}(\mathbf Z_w)\bigr)$
of $(\upphi_v)^{-1}(U_v\cap U_w)$. Then $\mathring{\mathfrak q}_n$ is
generated by the elements $b_\ell$ for $\ell\in P(w)$ and $\mathfrak q$
is the preimage of $\mathring{\mathfrak q}_n$ under the canonical map
$A_n\to B_n$. In other words, $\mathfrak q$ is the saturation with
respect to $\mathscr S_n$ of the ideal of $A_n$ generated by the numerators
of the elements $b_\ell$ for $\ell\in P(w)$. Though algorithmically doable
in any concrete example, finding the saturation is a demanding calculation,
which we will bypass by replacing $\mathfrak q$ by an approximation
$\widetilde{\mathfrak q}_n$.

\subsection{Inclusion and multiplicity, I}
\label{ss:IncMulI}
This section is devoted to the situation where the paths representing $v$
and $w$ stay parallel to each other at distance~two; specifically, we
assume that $v(\ell)=w(\ell)$ for each $\ell\in\{2,\ldots,n-1\}$ and
$(v(n),w(n))=(-,+)$.

\begin{proposition}
\label{pr:IncMulI}
Under these assumptions:

\vspace{-12pt}
\begin{enumerate}
\item
\label{it:PrIMIa}
The inclusion
$\mathcal Y(\mathbf Z_v)\subset\mathcal X(\mathbf Z_w,(1,n-1))$
holds if and only if the last latter of $w'$ is significant.
\item
\label{it:PrIMIb}
If the condition in \ref{it:PrIMIa} is fulfilled, then the multiplicity
of $\mathcal Y(\mathbf Z_v)$ in the intersection product
$\mathcal X(\mathbf Z_w,(1,n-1))\cdot\,\BDConv_n^{\bm\lambda}
\bigl|_{\Delta}$ is equal to one.
\end{enumerate}
\end{proposition}

The proof of Proposition~\ref{pr:IncMulI} fills the remainder of this
section.

Let us denote by $S(v)$ the set of all positions $\ell\in\{1,\ldots,n\}$
such that the letter $v(\ell)$ is significant in~$v$.

In agreement with the convention set forth in sect.~\ref{ss:FindEqns},
we define $A_\ell=\mathbb C[x_2][x,a_1,\ldots,a_\ell]$ for each
$\strut\ell\in\{1,\ldots,n\}$, where $x=x_1-x_2$. We rewrite the
indeterminate $z$ as $\tilde z+x_2$. We~set $\widetilde P_1=\tilde z-x$
and $\widetilde Q_1=a_1$. For $\strut\ell\in\{2,\ldots,n-1\}$, we define
by induction two polynomials~$\widetilde P_\ell$, $\widetilde Q_\ell$
in $A_\ell[\tilde z]$ as follows:
\begin{itemize}
\item
If $v(\ell)=w(\ell)=+$ and $\ell\in S(v)$, then
$$\widetilde P_\ell=\widetilde P_{\ell-1}\quad\;\text{and}\quad\;
\widetilde Q_\ell=\frac{a_\ell\widetilde P_{\ell-1}+\widetilde Q_{\ell-1}
-\bigl(a_\ell\widetilde P_{\ell-1}+\widetilde Q_{\ell-1}\bigr)
\bigl(0\bigr)}{\tilde z}.$$
\item
If $v(\ell)=w(\ell)=+$ and $\ell\notin S(v)$, then
$\widetilde P_\ell=\widetilde P_{\ell-1}$ and
$\widetilde Q_\ell=\bigl(\widetilde Q_{\ell-1}-
\widetilde Q_{\ell-1}\bigl(0\bigr)\bigr)/{\tilde z}$.
\item
If $v(\ell)=w(\ell)=-$, then
$\widetilde P_\ell=\widetilde P_{\ell-1}+a_\ell\widetilde Q_{\ell-1}$ and
$\widetilde Q_\ell=\tilde z\,\widetilde Q_{\ell-1}$.
\end{itemize}
Moreover, in the case where $v(\ell)=w(\ell)=+$, set
$$\widetilde c_\ell=\begin{cases}
\bigl(a_\ell\widetilde P_{\ell-1}+\widetilde Q_{\ell-1}\bigr)
\bigl(0\bigr)&\text{if $\ell\in S(v)$,}\\[4pt]
\;a_\ell&\text{otherwise,}
\end{cases}$$
and set
$$\widetilde c_n=\bigl(\widetilde P_{n-1}+a_n\widetilde Q_{n-1}\bigr)
\bigl(0\bigr).$$

\begin{other}{Remark}
\label{rk:DependVar}
The polynomials $\widetilde P_\ell$ and $\widetilde Q_\ell$ do not depend
on the variables $a_j$ with $j\in P(v)\setminus S(v)$. The elements
$\widetilde c_\ell$ for $\ell\in\{2,\ldots,n-1\}\cap P(v)\cap S(v)$ and
$\widetilde c_n$ enjoy the same property.
\end{other}

\vspace*{-4pt}
For $\ell\in\{1,\ldots,n\}$:

\vspace*{-10pt}
\begin{itemize}
\item
let $\mathring{\mathfrak q}_\ell$ be the ideal of
$B_\ell$ generated by $\{b_j\mid j\in P(w),\;j\leq\ell\}$;
\item
let $\widetilde{\mathfrak q}_\ell$ be the ideal of $A_\ell$ generated
by $\{\widetilde c_j\mid j\in P(w),\;j\leq\ell\}$;
\item
let $d_\ell$ be the weight of the word $v(1)v(2)\cdots v(\ell)$ and
set $D_\ell=\max(d_1,d_2,\ldots,d_\ell)$.
\end{itemize}

As noticed before, a $+$ letter at position $\ell$ in $v$ is
significant if and only if $\ell$ marks the first time that the path
representing $v$ reaches a new height; agreeing that $D_0=0$, this
translates to
$$\ell\in P(v)\cap S(v)\;\Longleftrightarrow\;d_\ell>D_{\ell-1}.$$
For the record, we also note that the last letter of $w'$ is significant
if and only if $d_{n-1}=D_{n-1}$.

\begin{lemma}
\label{le:Induc}
For $\ell\in\{1,\ldots,n-1\}$, we have

\vspace{-12pt}
\renewcommand\theenumi{(\roman{enumi})${}_\ell$}
\begin{enumerate}
\item
$\mathscr S_\ell^{-1}\widetilde{\mathfrak q}_\ell=\mathring{\mathfrak q}_\ell$,
\item
$\widetilde P_\ell(\tilde z)\equiv P_\ell(z)
\pmod{\mathring{\mathfrak q}_\ell[z]}\;$ and
$\;\widetilde Q_\ell(\tilde z)\equiv Q_\ell(z)
\pmod{\mathring{\mathfrak q}_\ell[z]}$,
\item
$\tilde z^{D_\ell-d_\ell}$ divides $\widetilde Q_\ell$.
\end{enumerate}
\renewcommand\theenumi{(\alph{enumi})}
\end{lemma}
\begin{proof}
We proceed by induction on $\ell$. The statements are banal for $\ell=1$.
Suppose that $2\leq\ell\leq n-1$ and that statements
(i)${}_{\ell-1}$, (ii)${}_{\ell-1}$ and (iii)${}_{\ell-1}$ hold.

Suppose first that $(v(\ell),w(\ell))=(+,+)$. Then by construction
\begin{gather}
\label{eq:Induc1}
b_\ell=\bigl(a_\ell P_{\ell-1}+Q_{\ell-1}\bigr)\bigl(x_2\bigr)\times
f_\ell^{-1},\\[4pt]
\label{eq:Induc2}
P_\ell=P_{\ell-1}-b_\ell R_{\ell-1},\qquad Q_\ell=\frac{a_\ell
P_{\ell-1}+Q_{\ell-1}-b_\ell S_\ell}{z-x_2}.
\end{gather}

If $\ell\notin S(v)$, then $d_{\ell-1}+1=d_\ell\leq D_{\ell-1}$,
and we see by (iii)${}_{\ell-1}$ that $\widetilde Q_{\ell-1}(0)=0$.
Using (ii)${}_{\ell-1}$, we deduce that
$Q_{\ell-1}(x_2)\in\mathring{\mathfrak q}_{\ell-1}$.
On the other hand, the matrix
$\begin{pmatrix}P_{\ell-1}(x_2)&Q_{\ell-1}(x_2)\\
R_{\ell-1}(x_2)&S_{\ell-1}(x_2)\end{pmatrix}$
with coefficients in $B_{\ell-1}$ has determinant one. After reduction
modulo $\strut\mathring{\mathfrak q}_{\ell-1}$, the coefficient in the top
right corner becomes zero; it follows that $P_{\ell-1}(x_2)$ is invertible
in the quotient ring $B_{\ell-1}/\mathring{\mathfrak q}_{\ell-1}$.
Reducing~\eqref{eq:Induc1} modulo $\mathring{\mathfrak q}_{\ell-1}B_\ell$
and noting that here $\widetilde c_\ell=a_\ell$, we deduce that $b_\ell$
and $\widetilde c_\ell$ generate the same ideal in
$B_\ell/\mathring{\mathfrak q}_{\ell-1}B_\ell$.
This piece of information allows to deduce (i)${}_\ell$
from~(i)${}_{\ell-1}$. From~\eqref{eq:Induc2} and the fact that
$a_\ell\in\mathring{\mathfrak q}_\ell$, we get
$$P_\ell\equiv P_{\ell-1}\pmod{\mathring{\mathfrak q}_\ell[z]},\qquad
Q_\ell\equiv\frac{Q_{\ell-1}-Q_{\ell-1}(x_2)}{z-x_2}
\pmod{\mathring{\mathfrak q}_\ell[z]}.$$
Then (ii)${}_\ell$ and (iii)${}_\ell$ follow from
(ii)${}_{\ell-1}$ and (iii)${}_{\ell-1}$ and from the
definition of $\widetilde P_\ell$ and $\widetilde Q_\ell$.

If $\ell\in S(v)$, then \eqref{eq:Induc1} and (ii)${}_{\ell-1}$
lead to $b_\ell\equiv\widetilde c_\ell/f_\ell$ modulo
$\mathring{\mathfrak q}_{\ell-1}B_\ell$. Again, $b_\ell$ and
$\widetilde c_\ell$ generate the same ideal in
$B_\ell/\mathring{\mathfrak q}_{\ell-1}B_\ell$, so we can deduce
(i)${}_\ell$ from~(i)${}_{\ell-1}$. Then (ii)${}_\ell$ follows from
(ii)${}_{\ell-1}$ and~\eqref{eq:Induc2}. Also, (iii)${}_{\ell-1}$
holds trivially since $D_\ell=d_\ell$.

It remains to tackle the case $(v(\ell),w(\ell))=(-,-)$. Here
(i)${}_\ell$, (ii)${}_\ell$ and (iii)${}_\ell$ can be deduced from
(i)${}_{\ell-1}$, (ii)${}_{\ell-1}$ and (iii)${}_{\ell-1}$ without ado.
\end{proof}

\begin{lemma}
With the notation above,
$$\mathscr S_n^{-1}\widetilde{\mathfrak q}_n=\mathring{\mathfrak q}_n
\quad\;\text{and}\quad\;
\mathfrak q=\bigl\{g\in A_n\bigm|\exists f\in\mathscr S_n,\;fg\in
\widetilde{\mathfrak q}_n\bigr\}.$$
\end{lemma}
\begin{proof}
From $(v(n),w(n))=(-,+)$, we deduce
$$b_n=\bigl(P_{n-1}+a_nQ_{n-1}\bigr)\bigl(x_2\bigr)\times f_n^{-1}.$$
From the assertion (ii)${}_{n-1}$ in Lemma~\ref{le:Induc},
we deduce that $b_n\equiv\widetilde c_n/f_n$ modulo
$\mathring{\mathfrak q}_{n-1}B_n$. Thus, $b_n$ and
$\widetilde c_n$ generate the same ideal in
$B_n/\mathring{\mathfrak q}_{n-1}B_n$, and from the
assertion~(i)${}_{n-1}$ in Lemma~\ref{le:Induc}, we conclude that
$\mathscr S_n^{-1}\widetilde{\mathfrak q}_n=\mathring{\mathfrak q}_n$.
The second announced equality then follows from the definition of
$\mathfrak q$ as the preimage of $\mathring{\mathfrak q}_n$ under the
canonical map $A_n\to B_n$, with $B_n=\mathscr S_n^{-1}A_n$.
\end{proof}

\begin{lemma}
\label{le:Exclus}
If the last letter of $w'$ is not significant,
then $\mathring{\mathfrak q}_n=B_n$.
\end{lemma}
\begin{proof}
Assume that the last letter of $w'$ is not significant. Then
$D_{n-1}-d_{n-1}\geq1$, and by assertion (iii)${}_{n-1}$ in
Lemma~\ref{le:Induc}, we get $\widetilde Q_{n-1}(0)=0$.
Using assertion (ii)${}_{n-1}$ in that lemma, we deduce that
$Q_{n-1}(x_2)\in\mathring{\mathfrak q}_{n-1}$. Since the
matrix $\begin{pmatrix}P_{n-1}(x_2)&Q_{n-1}(x_2)\\
R_{n-1}(x_2)&S_{n-1}(x_2)\end{pmatrix}$ has determinant~$1$,
we see that $P_{n-1}(x_2)$ is invertible in the ring
$B_{n-1}/\mathring{\mathfrak q}_{n-1}$. Then
$b_n=\bigl(P_{n-1}+a_nQ_{n-1}\bigr)\bigl(x_2\bigr)\times f_n^{-1}$
is invertible in $B_n/\mathring{\mathfrak q}_{n-1}B_n$, and we conclude
that $\mathring{\mathfrak q}_n=B_n$.
\end{proof}

Lemma~\ref{le:Exclus} asserts that if the last letter of $w'$ is not
significant, then $U_v\cap\dot{\mathcal X}(\mathbf Z_w)=\varnothing$,
and thus $U_v\cap\mathcal X(\mathbf Z_w,(1,n-1))=\varnothing$.
Since $U_v$ contains $\mathcal Y(\mathbf Z_v)$, this proves half of
Proposition~\ref{pr:IncMulI}~\ref{it:PrIMIa}.

For the rest of this section, we assume that the last letter of $w'$
is significant. We want to show that $\mathcal Y(\mathbf Z_v)$ is
contained in $\mathcal X(\mathbf Z_w,(1,n-1))$. It would be rather easy
to prove the inclusion $\widetilde{\mathfrak q}_n\subset\mathfrak p$,
but this would not be quite enough, since we do not know that
$\widetilde{\mathfrak q}_n=\mathfrak q$. (We believe that this
equality is correct but we are not able to prove it.) Instead we will
look explicitly at the zero set of $\widetilde{\mathfrak q}_n$ in the
neighborhood of $(\upphi_v)^{-1}(\mathcal Y(\mathbf Z_v))$. This zero
set is the algebraic subset of $(\upphi_v)^{-1}(U_v)$ defined by the
equations $\widetilde c_\ell$ for $\ell\in P(w)$.

Our analysis is pedestrian. We observe that there are two kinds of
equations~$\widetilde c_\ell$. When $\ell\in P(v)\setminus S(v)$, the
equation $\widetilde c_\ell$ reduces to the variable $a_\ell$; this
equation and variable can simply be discarded because $a_\ell$ is an
equation for $\mathcal Y(\mathbf Z_v)$ as well. The other equations
involve the other variables.

Set $D=D_n$. The map $\ell\mapsto d_\ell$ is an increasing bijection
from $P(v)\cap S(v)$ onto $\{1,\ldots,D\}$. We define $L$ as the
largest element in $P(v)\cap S(v)$; then $L$ is the smallest element
in $\{\ell\mid d_\ell=D\}$. For $\ell\in\{1,\ldots,n\}$, we denote
by $\ell^-$ the largest element in $\{1,\ldots,\ell\}\cap P(v)\cap S(v)$.
In partic\-ular,~$\ell^-=\ell$ if $\ell\in P(v)\cap S(v)$ and $\ell^-=L$
if $\ell\geq L$; also $d_{\ell^-}=D_\ell$.

Given $\ell\in\{1,\ldots,n\}$, let $\sigma_\ell$ be the sum of the
variables $a_j$ for $j\in\{2,\ldots,\ell\}$ such that $v(j)=-$ and
$d_{j-1}=D$; thus $\sigma_\ell=0$ if $\ell\leq L$.

We define a graduation on $A_n$ by setting $\deg x=1$,
$\deg a_\ell=D+1-d_\ell$ for $\ell\in P(v)\cap S(v)$, and $\deg a_\ell=0$
for the other variables. For $d\geq1$, we denote by $J_d$ the ideal
of $A_n$ spanned by monomials of degree at least $d$.

\begin{lemma}
\label{le:ValTildeP}
Let $\ell\in\{1,\ldots,n-1\}$.

\vspace{-12pt}
\renewcommand\theenumi{(\roman{enumi})${}_\ell$}
\begin{enumerate}
\item
If $\ell\leq L$, then $\widetilde P_\ell(\tilde z)\equiv\tilde z-x
\pmod{J_2[\tilde z]}$; if $\ell\geq L$, then
$\widetilde P_\ell(0)\equiv a_L\sigma_\ell-x\pmod{J_2}$.
\item
$\widetilde Q_\ell(\tilde z)\equiv\tilde z^{D_\ell-d_\ell}\,a_{\ell^-}
\pmod{J_{D+2-d_{\ell^-}}[\tilde z]}$.
\end{enumerate}
\end{lemma}
\begin{proof}
The proof starts with a banal verification for $\ell=1$ and then
proceeds by induction on~$\ell$. Suppose that $2\leq\ell\leq n-1$
and that statements (i)${}_{\ell-1}$ and (ii)${}_{\ell-1}$ hold.

Assume first that $v(\ell)=w(\ell)=-$. Here (ii)${}_\ell$ is an
immediate consequence of (ii)${}_{\ell-1}$. If $\ell-1<L$, then
$d_{(\ell-1)^-}<D$, so $\deg a_{(\ell-1)^-}\geq2$, and
$\widetilde Q_{\ell-1}\in J_2[\tilde z]$ by statement
(ii)${}_{\ell-1}$. As a result,
$\strut\widetilde P_\ell\equiv\widetilde P_{\ell-1}\pmod{J_2[\tilde z]}$,
so (i)${}_\ell$ directly follows from (i)${}_{\ell-1}$.
If $\ell-1\geq L$, then either $d_{\ell-1}=D$, in which case
$\strut\widetilde Q_{\ell-1}(0)\equiv a_L\pmod{J_2}$ and
$\sigma_\ell=\sigma_{\ell-1}+a_\ell$, or $d_{\ell-1}<D$,
in which case $\strut\widetilde Q_{\ell-1}(0)\equiv0\pmod{J_2}$
and $\sigma_\ell=\sigma_{\ell-1}$. In both cases,
$\widetilde P_\ell(0)-(a_L\sigma_\ell)
\equiv\widetilde P_{\ell-1}(0)-(a_L\sigma_{\ell-1})\pmod{J_2}$,
and again (i)${}_\ell$ follows from (i)${}_{\ell-1}$.

Assume now that $v(\ell)=w(\ell)=+$ and that $\ell\in S(v)$.
Certainly then (i)${}_\ell$ is readily deduced from (i)${}_{\ell-1}$.
Further, we remark that $d_{(\ell-1)^-}=d_{\ell^-}-1$, so
$\deg a_{(\ell-1)^-}=D+2-d_{\ell^-}$, hence $\widetilde Q_{\ell-1}$
is zero modulo $J_{D+2-d_{\ell^-}}[\tilde z]$ by (ii)${}_{\ell-1}$.
Using (i)${}_{\ell-1}$, we conclude that
$\widetilde Q_\ell\equiv a_\ell\pmod{J_{D+2-d_{\ell^-}}[\tilde z]}$,
so (ii)${}_\ell$ holds.

The third situation, namely $v(\ell)=w(\ell)=+$ and $\ell\notin S(v)$,
presents no difficulties.
\end{proof}

\begin{lemma}
\label{le:ValTildeC}
\mbox{}

\vspace{-12pt}
\begin{enumerate}
\item
\label{it:LeVTCa}
For $\ell\in\{2,\ldots,n-1\}\cap P(v)\cap S(v)$, we have
$\widetilde c_\ell\equiv-a_\ell\,x+ a_{(\ell-1)^-}\pmod{J_{D+3-d_\ell}}$.
\item
\label{it:LeVTCb}
We have $\widetilde c_n\equiv a_L\sigma_n-x\pmod{J_2}$.
\end{enumerate}
\end{lemma}
\begin{proof}
Let $\ell\in\{2,\ldots,n-1\}\cap P(v)\cap S(v)$. Then
$D_{\ell-1}=d_{\ell-1}$ and $d_{(\ell-1)^-}=d_\ell-1$. By
Lemma~\ref{le:ValTildeP}, $\widetilde P_{\ell-1}(0)\equiv-x\pmod{J_2}$ and
$\widetilde Q_{\ell-1}(0)\equiv a_{(\ell-1)^-}\pmod{J_{D+3-d_\ell}}$.
This gives~\ref{it:LeVTCa}.

Since the last letter of $w'$ is assumed to be significant,
we have $d_{n-1}=D_{n-1}=D$, so $\sigma_n=\sigma_{n-1}+a_n$.
From Lemma~\ref{le:ValTildeP}, we get
$\widetilde P_{n-1}(0)\equiv a_L\sigma_{n-1}-x\pmod{J_2}$ and
$\widetilde Q_{n-1}(0)\equiv a_L\pmod{J_2}$. This gives~\ref{it:LeVTCb}.
\end{proof}

\begin{lemma}
\label{le:Elimin}
There exists an element $\widetilde g\in A_n$, which depends only on the
variables $x$, $a_1$, and $a_j$ with $v(j)=-$, such that
\begin{alignat}2
\label{eq:CongTildeG}
\widetilde g&\equiv
\widetilde c_n\,x^{D-1}\times\prod_{\substack{\ell\in P(v)\cap S(v)\\\ell\geq2}}
\Bigl(-\widetilde P_{\ell-1}(0)\Bigr)^{p_\ell}&&
\pmod{\widetilde{\mathfrak q}_L}\\[8pt]
\label{eq:NewtonVert}
\widetilde g&\equiv x^q\,\bigl(a_1\sigma_n-x^D\bigr)&&\pmod{J_{q+D+1}}
\end{alignat}
where each $p_\ell$ and $q$ are nonnegative integers.
\end{lemma}
\begin{proof}
Consider
$$\widetilde g_L=\widetilde c_n\,x^{D-1}+\sum_{\substack{\ell\in P(v)
\cap S(v)\\[2pt]\ell\geq2}}\widetilde c_\ell\,\sigma_n\,x^{d_\ell-2}.$$
An immediate calculation based on Lemma~\ref{le:ValTildeC} yields
$$\widetilde g_L\equiv a_1\sigma_n-x^D\pmod{J_{D+1}}.$$
This $\widetilde g_L$ meets the specifications for $\widetilde g$
(with $p_\ell$ and $q$ all equal to zero) except that it may involve
other variables than those prescribed.

We are not bothered by the variables $a_j$ for $j\in P(v)\setminus S(v)$
because $\widetilde g_L$ do not depend on them (see Remark~\ref{rk:DependVar}).
The variables $x$ and $a_j$ with $v(j)=-$ are allowed. The only trouble comes
then from the variables $a_j$ with $j\in\{2,\ldots,n-1\}\cap P(v)\cap S(v)$.
We will eliminate them in turn.

Assume that $L\geq2$. Let $\ell\in\{2,\ldots,n-1\}\cap P(v)\cap S(v)$
and assume that we succeeded in constructing an element
$\widetilde g_\ell\in\widetilde{\mathfrak q}_n$ which satisfies
 \eqref{eq:CongTildeG} and \eqref{eq:NewtonVert} and depends only
on the variables $x$ and $a_j$ with $v(j)=-$ or $j\leq\ell$.
Expand $\widetilde g_\ell$ as a polynomial in $a_\ell$
$$\widetilde g_\ell=\sum_{s=0}^rh_s\,a_\ell^s$$
where the coefficients $h_s$ only depend on $x$ and on the variables
$a_j$ such that $v(j)=-$ or $j<\ell$. Then define
$$\widetilde g_{(\ell-1)^-}=\sum_{s=0}^rh_s\,
\Bigl(-\widetilde P_{\ell-1}(0)\Bigr)^{r-s}\,
\Bigl(\widetilde Q_{\ell-1}(0)\Bigr)^s.$$
This $\widetilde g_{(\ell-1)^-}$ only involves the variables $x$ and
$a_j$ with $v(j)=-$ or $j\leq\ell-1$. In fact, we can strengthen
the latter inequality to $j\leq(\ell-1)^-$ because
$\widetilde g_{(\ell-1)^-}$ does not depend on the variables
$a_j$ with $j\in P(v)\setminus S(v)$. Moreover,
$\widetilde g_{(\ell-1)^-}$ also satisfies  \eqref{eq:CongTildeG}
and \eqref{eq:NewtonVert}, but for different integers than
$\widetilde g_\ell$: one has to increase $p_\ell$ and $q$ by $r$.
(To verify that $\widetilde g_{(\ell-1)^-}$ satisfies
\eqref{eq:NewtonVert} with $q+r$ instead of $q$, one observes that
\begin{alignat*}2
h_0&\equiv x^q\,\bigl(a_1\sigma_n-x^D\bigr)&&(\bmod\,\,J_{q+D+1})\\[4pt]
h_s&\in J_{q+D+1-s(D+1-d_\ell)}&\quad&\text{for each
$s\in\{1,\ldots,r\}$}
\end{alignat*}
and uses Lemma~\ref{le:ValTildeP}.)

At the end of the process, we obtain an element
$\widetilde g=\widetilde g_1$ which enjoys the desired properties.
\end{proof}

Let us recall a few important points:

\vspace{-12pt}
\begin{itemize}
\item
$A_n=\mathbb C[x_2][x,a_1,\ldots,a_n]$ is the coordinate ring
of $(\upphi_v)^{-1}(U_v)$. The variable $x_2$ is dumb (no equations
depend on it); we get rid of it by specializing it to an arbitrary value.
\item
The ring $B_1$ is $\mathbb C[x_2]\bigl[x,a_1,f_1^{-1}\bigr]$
with $f_1=a_1$. For $\ell\geq2$, we produce an explicit function
$f_\ell\in B_{\ell-1}[a_\ell]$ and we set
$B_\ell=B_{\ell-1}\bigl[a_\ell,f_\ell^{-1}\bigr]$. The ring $B_n$
is the coordinate ring of~$(\upphi_v)^{-1}(U_v\cap U_w)$.
\item
$\mathscr S_n$ is a finitely generated multiplicative subset of
$A_n$ such that $B_n=\mathscr S_n^{-1}A_n$.
\item
Polynomials $\widetilde c_\ell\in A_\ell$ are defined for each
$\ell\in P(w)$. The ideal of $A_n$ generated by these elements is
denoted by $\widetilde{\mathfrak q}_n$.
\item
The ideal $\mathfrak p\subset A_n$ of
$(\upphi_v)^{-1}(\mathcal Y(\mathbf Z_v))$
is generated by the variables $x$ and $a_\ell$ for $\ell\in P(v)$.
\item
The ideal $\mathfrak q\subset A_n$ of
$(\upphi_v)^{-1}(U_v\cap\mathcal X(\mathbf Z_w,(1,n-1)))$
is the saturation of $\widetilde{\mathfrak q}_n$ with respect to
$\mathscr S_n$.
\item
$\sigma_1$, \dots, $\sigma_n$ are certain sums of variables $a_\ell$
with $v(\ell)=-$; these linear forms are not pairwise distinct, but
$\sigma_n$ differs from all the other ones, for only it involves $a_n$.
\end{itemize}

\begin{lemma}
\label{le:GermCurv}
Fix $\alpha_\ell\in\mathbb C$ for each $\ell\in\{1,\ldots,n\}\setminus
P(v)$ such that, when $a_\ell$ is assigned the value $\alpha_\ell$,
the linear form $\sigma_n$ takes a value different from all the other
$\sigma_j$. Consider these numbers $\alpha_\ell$ as constant
functions of the variable $\xi$. Set also $\alpha_\ell=0$ for
$\ell\in P(v)\setminus S(v)$. Then there exists a neighborhood $\Omega$
of $0$ in $\mathbb C$ and analytic functions
$\alpha_\ell:\Omega\to\mathbb C$ for $\ell\in P(v)\cap S(v)$ such that
\begin{enumerate}
\item
\label{it:LeGCa}
If $\ell\in P(v)\cap S(v)$, then $\alpha_\ell(\xi)\sim
\xi^{D+1-d_\ell}/\sigma_n$.
\item
\label{it:LeGCb}
The point $(\xi,\alpha_1(\xi),\ldots,\alpha_n(\xi))$
belongs to the zero locus of\/ $\widetilde{\mathfrak q}_n$ for each
$\xi\in\Omega$.
\item
\label{it:LeGCc}
The point $(\xi,\alpha_1(\xi),\ldots,\alpha_n(\xi))$ belongs to
$U_w$ for each $\xi\neq0$ in $\Omega$.
\end{enumerate}
\end{lemma}
\begin{proof}
Let $\widetilde g$ be as in Lemma~\ref{le:Elimin}.
We consider that the variables $a_\ell$ with $\ell>1$ occuring in
$\widetilde g$ are assigned the values $\alpha_\ell$ fixed in the
statement of the lemma. We can then regard $\widetilde g$ as a polynomial
in the indeterminates $x$ and $a_1$ with complex coefficients,
or as a polynomial in the indeterminate $a_1$ with coefficients in the
valued field $\mathbb C(\!(x)\!)$. Equation \eqref{eq:NewtonVert} shows that
the points $(0,D+q)$ and $(1,q)$ are vertices of the Newton polygon of
$\widetilde g$. Therefore $\widetilde g$ admits a unique root of valuation
$D$ in $\mathbb C(\!(x)\!)$, which we denote by $\alpha_1$, and the power
series $\alpha_1$ has a positive radius of convergence. Proceeding by
induction on $\ell\in\{2,\ldots,n-1\}\cap P(v)\cap S(v)$, and solving the
equation $\widetilde c_\ell=0$, we define
\begin{equation}
\label{eq:GermCurv}
\alpha_\ell(\xi)=-\widetilde Q_{\ell-1}(0)/\widetilde P_{\ell-1}(0),
\end{equation}
where the right-hand side is evaluated at
$(\xi,\alpha_1(\xi),\ldots,\alpha_{\ell-1}(\xi))$; this is a well-defined
process and $\alpha_\ell(\xi)$ satisfies the equivalent given in the
statement, because Lemma~\ref{le:ValTildeP} guarantees that after
evaluation
$$\widetilde P_{\ell-1}(0)=-\xi+O\bigl(\xi^2\bigr)\quad\text{and}\quad
\widetilde Q_{\ell-1}(0)=\alpha_{(\ell-1)^-}(\xi)+
O\Bigl(\xi^{D+2-d_{(\ell-1)^-}}\Bigr),$$
so the denominator in~\eqref{eq:GermCurv} does not vanish if
$\xi\neq0$. Moreover, \eqref{eq:CongTildeG} ensures that the
equation $\widetilde c_n=0$ is enforced too. Therefore this
construction gives~\ref{it:LeGCa} and \ref{it:LeGCb}.

We will prove~\ref{it:LeGCc} by showing that none of the
functions $f_\ell$ vanish when evaluated on the point
$(\xi,\alpha_1(\xi),\ldots,\alpha_n(\xi))$ with $\xi\neq0$.
This is true for $\ell=1$, because $f_1=a_1$ and
$\alpha_1(\xi)\sim\xi^D/\sigma_n$. We assume known that $f_1$,
\dots, $f_{\ell-1}$ do not vanish on our germ of curve.

\vspace{-8pt}
\begin{itemize}
\item
In the case $(v(\ell),w(\ell))=(+,+)$, we have
$$f_\ell=\bigl(a_\ell R_{\ell-1}+S_{\ell-1}\bigr)\bigl(x_2\bigr).$$
The congruences in Lemma~\ref{le:Induc} allow to rewrite the
equation $\widetilde c_\ell=0$ in the form
$$\bigl(a_\ell P_{\ell-1}+Q_{\ell-1}\bigr)\bigl(x_2\bigr)=0;$$
this is satisfied after evaluation at the point
$(\xi,\alpha_1(\xi),\ldots,\alpha_n(\xi))$. Using then the relation
$\bigl(P_{\ell-1}S_{\ell-1}-Q_{\ell-1}R_{\ell-1}\bigr)\bigl(x_2\bigr)=1$,
we obtain
$$P_{\ell-1}(x_2)\times f_\ell=P_{\ell-1}(x_2)\bigl(a_\ell R_{\ell-1}+S_{\ell-1}\bigr)
\bigl(x_2\bigr)=1+R_{\ell-1}(x_2)\bigl(a_\ell P_{\ell-1}+Q_{\ell-1}\bigr)
\bigl(x_2\bigr)=1.$$
Thus, $f_\ell$ does not vanish at
$(\xi,\alpha_1(\xi),\ldots,\alpha_n(\xi))$.
\item
The case $(v(\ell),w(\ell))=(-,+)$, that is $\ell=n$, is amenable to
a similar treatment.
\item
The remaining case is $(v(\ell),w(\ell))=(-,-)$. Here by
Lemma~\ref{le:Induc} we have after substitution
$$f_\ell=\bigl(P_{\ell-1}+a_\ell Q_{\ell-1}\bigr)\bigl(x_2\bigr)
=\bigl(\widetilde P_{\ell-1}+a_\ell\widetilde Q_{\ell-1}\bigr)\bigl(0\bigr),$$
and by Lemma~\ref{le:ValTildeP} and the equivalence in~\ref{it:LeGCa}
$$\widetilde P_{\ell-1}(0)=(\sigma_{\ell-1}/\sigma_n-1)\,\xi
+O\bigl(\xi^2\bigr)\ \;\text{and}\ \;\widetilde Q_{\ell-1}(0)=
\begin{cases}
\xi/\sigma_n+O(\xi^2)&\text{if $d_{\ell-1}=D_{\ell-1}=D$,}\\
O(\xi^2)&\text{otherwise.}
\end{cases}$$
Therefore $f_\ell$ is equivalent to $(\sigma_\ell/\sigma_n-1)\,\xi$.
Shrinking $\Omega$ if necessary, we can ensure that $f_\ell$ does not
vanish.
\end{itemize}

\vspace{-8pt}
This concludes the induction and establishes~\ref{it:LeGCc}.
\end{proof}

To sum up, we construct a germ of smooth algebraic curve contained in
the zero locus of $\widetilde{\mathfrak q}_n$. The ideal of this curve
is a prime ideal of $A_n$ which contains $\widetilde{\mathfrak q}_n$
and is disjoint from $\mathscr S_n$; hence it contains $\mathfrak q$.
As a result, our curve is contained in
$(\upphi_v)^{-1}(U_v\cap\mathcal X(\mathbf Z_w,(1,n-1)))$. The point
at $\xi=0$ of this curve has for coordinates the values $\alpha_\ell$
chosen for each $\ell\in\{1,\ldots,n\}\setminus P(v)$, contingent on
$\sigma_n\neq\sigma_j$ for $j\in\{1,\ldots,n-1\}$, the other coordinates
being zero. Such points form an open dense subset of
$(\upphi_v)^{-1}(\mathcal Y(\mathbf Z_v))$, so we conclude that
$\mathcal Y(\mathbf Z_v)\subset\mathcal X(\mathbf Z_w,(1,n-1))$.
This proves the missing half of Proposition~\ref{pr:IncMulI}~\ref{it:PrIMIa}
(the first half was obtained just after Lemma~\ref{le:Exclus}).

As a consequence, $\mathfrak q\subset\mathfrak p$. To ease the reading
of the sequel, we will omit the subscripts $n$ in the notation $A_n$ and
$\widetilde{\mathfrak q}_n$. For $\ell\in\{1,\ldots,n\}$, we set
$R(\ell)=\bigl\{j\in\{2,\ldots,\ell\}\bigm|v(j)=-,\;d_{j-1}=D_{j-1}\bigr\}$.

\begin{lemma}
\label{le:PrepNaka}
\begin{enumerate}
\item
\label{it:LePNa}
For each $\ell\in\{1,\ldots,n-1\}$, we have
\begin{align*}
&\widetilde P_\ell\equiv\tilde z\pmod{\mathfrak p[\tilde z]},\qquad
\widetilde Q_\ell\equiv\tilde z^{D_\ell-d_\ell}\,a_{\ell^-}
\pmod{\mathfrak p^2[\tilde z]},\\[4pt]
&\widetilde P_\ell(0)\equiv-x+\sum_{j\in R(\ell)}a_{(j-1)^-}a_j
\pmod{\mathfrak p^2}.
\end{align*}
\item
\label{it:LePNb}
In the local ring $A_{\mathfrak p}$, we have
$\mathfrak pA_{\mathfrak p}=xA_{\mathfrak p}+\mathfrak qA_{\mathfrak p}
+\mathfrak p^2A_{\mathfrak p}$.
\end{enumerate}
\end{lemma}
\begin{proof}
Statement \ref{it:LePNa} is proved by a banal induction. Let us
tackle \ref{it:LePNb}.

If $\ell\in P(v)\setminus S(v)$, then $a_\ell=\widetilde c_\ell$
belongs to $\widetilde{\mathfrak q}$.

If $\ell\in(P(v)\cap S(v))\setminus\{L\}$, then there exists
$m\in P(v)\cap S(v)$ such that $d_\ell=d_m-1$. Then $\ell=(m-1)^-$ and
$D_{m-1}=d_{m-1}$, whence by statement \ref{it:LePNa}
$$a_\ell\equiv\widetilde Q_{m-1}(0)=\widetilde c_m-a_m\widetilde P_{m-1}(0)
\equiv\widetilde c_m\pmod{\mathfrak p^2},$$
and therefore $a_\ell\in\widetilde{\mathfrak q}+\mathfrak p^2$.

Surely $D_{n-1}=d_{n-1}=D$ and $L=(n-1)^-$, so again by
statement \ref{it:LePNa}, we have
$$\widetilde c_n=\widetilde P_{n-1}(0)+a_n\widetilde Q_{n-1}(0)\equiv
\widetilde P_{n-1}(0)+a_La_n\equiv-x+\sum_{j\in R(n)}a_{(j-1)^-}a_j
\pmod{\mathfrak p^2}.$$
In the last sum, we gather the terms with the same value $\ell$ for
$(j-1)^-$: denoting by $\tau_\ell$ the sum of the variables $a_j$ for
$j\in\{2,\ldots,n\}$ such that $v(j)=-$ and $d_{j-1}=D_{j-1}=d_\ell$,
we obtain
$$\widetilde c_n\equiv-x+\sum_{\ell\in P(v)\cap S(v)}a_\ell\,\tau_\ell
\pmod{\mathfrak p^2}.$$
Noting that $a_\ell\in\widetilde{\mathfrak q}+\mathfrak p^2$ for
$\ell\in P(v)\cap S(v)\setminus\{L\}$ and that $\tau_L=\sigma_n$,
we get $a_L\sigma_n\in(x)+\widetilde{\mathfrak q}+\mathfrak p^2$.
Since $\sigma_n$ is invertible in $A_{\mathfrak p}$, we conclude that
$a_L\in xA_{\mathfrak p}+\widetilde{\mathfrak q}
A_{\mathfrak p}+\mathfrak p^2A_{\mathfrak p}$.

Altogether the remarks above show the inclusion
$$\mathfrak pA_{\mathfrak p}\subset xA_{\mathfrak p}
+\widetilde{\mathfrak q}A_{\mathfrak p}+\mathfrak p^2
A_{\mathfrak p}.$$
Joint with $\widetilde{\mathfrak q}\subset\mathfrak q\subset\mathfrak p$,
this gives statement \ref{it:LePNb}.
\end{proof}

The ideal in $A$ of the subvarieties
$$V=(\upphi_v)^{-1}(\mathcal Y(\mathbf Z_v))\quad\text{and}\quad
X=(\upphi_v)^{-1}(U_v\cap\mathcal X(\mathbf Z_w,(1,n-1)))$$
are $\mathfrak p$ and $\mathfrak q$, respectively. The local ring
$\mathscr O_{V,X}$ of $X$ along $V$ is the localization of
$\overline A=A/\mathfrak q$ at the ideal
$\overline{\mathfrak p}=\mathfrak p/\mathfrak q$.
Lemma~\ref{le:PrepNaka}~\ref{it:LePNb} combined with Nakayama's lemma
shows that the image of $x=x_1-x_2$ in $\overline A$ generates the
ideal $\overline{\mathfrak p}\,\overline A_{\overline{\mathfrak p}}$.
As a consequence, the order of vanishing of $x_1-x_2$ along $V$ is
equal to one, and by definition, this is the multiplicity of
$\mathcal Y(\mathbf Z_v)$ in the intersection product
$\mathcal X(\mathbf Z_w,(1,n-1))\cdot\,\BDConv_n^{\bm\lambda}\bigl|_\Delta$.
This proves Proposition~\ref{pr:IncMulI}~\ref{it:PrIMIb}.

\subsection{Inclusion, II}
\label{ss:IncMulII}
In this section, we again consider words $v$ and $w$ such that
$(v(1),w(1))=(+,-)$ and $\wt(v)=\wt(w)$ and explore the situation
where the path representing $v$ lies strictly above the one
representing $w$ (except of course at the two endpoints) but does
not stay parallel to it. We thus assume that there exists
$k\in\{2,\ldots,n-1\}$ such that $(v(k),w(k))=(+,-)$.

\begin{proposition}
\label{pr:IncMulII}
Under these assumptions,
$\mathcal Y(\mathbf Z_v)\not\subset\mathcal X(\mathbf Z_w,(1,n-1))$.
\end{proposition}

The proof of Proposition~\ref{pr:IncMulII} fills the remainder
of this section. Our argument is similar to our proof in
Proposition~\ref{pr:IncMulI}~\ref{it:PrIMIa}.

For each $\ell\in\{1,\ldots,n\}$, we define
$A_\ell=\mathbb C[x_2][x,a_1,\ldots,a_\ell]$,
where $x=x_1-x_2$. We introduce $\tilde z=z-x_2$.

In addition:

\vspace{-12pt}
\begin{itemize}
\item
let $K$ be the largest integer $k\in\{2,\ldots,n-1\}$ such that
$(v(k),w(k))=(+,-)$;
\item
for $\ell\in\{K,\ldots,n\}$, let $d_\ell$ be the weight of the word
$v(K+1)v(K+2)\cdots v(\ell)$, with the convention $d_K=0$;
\item
let $L$ be the smallest position $\ell>K$ such that
$(v(\ell),w(\ell))=(-,+)$ or $d_\ell>0$.
\end{itemize}

Set $\widetilde P_1=\tilde z-x$ and $\widetilde Q_1=a_1$.
For $\ell\in\{2,\ldots,L-1\}$, define by induction two polynomials
$\widetilde P_\ell$, $\widetilde Q_\ell$ in $A_\ell[z]$ as follows:
\begin{itemize}
\item
If $(v(\ell),w(\ell))=(+,+)$, then
$$\widetilde P_\ell=\widetilde P_{\ell-1}\quad\;\text{and}\quad\;
\widetilde Q_\ell=\begin{cases}
\displaystyle\frac{a_\ell\widetilde P_{\ell-1}+\widetilde Q_{\ell-1}
-\bigl(a_\ell\widetilde P_{\ell-1}+\widetilde Q_{\ell-1}\bigr)
\bigl(0\bigr)}{\tilde z}&\text{if $\ell<K$,}\\[8pt]
\displaystyle\frac{\widetilde Q_{\ell-1}-\widetilde Q_{\ell-1}(0)}{\tilde z}
&\text{if $\ell>K$.}
\end{cases}$$
\item
If $(v(\ell),w(\ell))=(-,+)$, then
$$\widetilde P_\ell=\frac{\widetilde P_{\ell-1}+a_\ell\widetilde
Q_{\ell-1}-\bigl(\widetilde P_{\ell-1}+a_\ell\widetilde Q_{\ell-1}\bigr)
\bigl(0\bigr)}{\tilde z}\quad\;\text{and}\quad\;
\widetilde Q_\ell=\widetilde Q_{\ell-1}.$$
\item
If $(v(\ell),w(\ell))=(+,-)$, then
$\widetilde P_\ell=\tilde z\,\widetilde P_{\ell-1}$ and
$\widetilde Q_\ell=a_\ell\widetilde P_{\ell-1}+\widetilde Q_{\ell-1}$.
\item
If $(v(\ell),w(\ell))=(-,-)$, then
$\widetilde P_\ell=\widetilde P_{\ell-1}+a_\ell\widetilde Q_{\ell-1}$ and
$\widetilde Q_\ell=\tilde z\,\widetilde Q_{\ell-1}$.
\end{itemize}

For $\ell\in\{1,\ldots,L\}$:

\vspace*{-10pt}
\begin{itemize}
\item
let $\mathring{\mathfrak q}_\ell$ be the ideal of
$B_\ell$ generated by $\{b_j\mid j\in P(w),\;j\leq\ell\}$;
\item
if $\ell\geq K$, let $\sigma_\ell$ be the sum of the $a_j$ for
$j\in\{K+1,\ldots,\ell\}$ such that $v(j)=-$ and $d_{j-1}=0$,
with the convention $\sigma_K=0$.
\end{itemize}

\begin{lemma}
\label{le:Induc2}
For $\ell\in\{1,\ldots,L-1\}$, we have

\vspace{-12pt}
\renewcommand\theenumi{(\roman{enumi})${}_\ell$}
\begin{enumerate}
\item
$\widetilde P_\ell(\tilde z)\equiv P_\ell(z)
\pmod{\mathring{\mathfrak q}_\ell[z]}\;$ and
$\;\widetilde Q_\ell(\tilde z)\equiv Q_\ell(z)
\pmod{\mathring{\mathfrak q}_\ell[z]}$,
\item
if $\ell\geq K$, then $\widetilde P_\ell(0)=\widetilde Q_K(0)\sigma_\ell$
and $\widetilde Q_\ell=\tilde z^{-d_\ell}\widetilde Q_K$.
\end{enumerate}
\renewcommand\theenumi{(\alph{enumi})}
\end{lemma}
\begin{proof}
One again proceeds by induction. The details are straightforward indeed,
except in the case where $(v(\ell),w(\ell))=(+,+)$ and $\ell>K$, where
one can follow the arguments offered in the proof of
Lemma~\ref{le:Induc} to get $a_\ell\in\mathring{\mathfrak q}_\ell$.
\end{proof}

We now distinguish three cases:
\begin{itemize}
\item
Assume that $d_{L-1}<0$. Then necessarily $(v(L),w(L))=(-,+)$. By
assertion~(ii)${}_{L-1}$ in Lemma~\ref{le:Induc2}, we get
$\widetilde Q_{L-1}(0)=0$. Using assertion (i)${}_{L-1}$ in that
lemma, we deduce that $Q_{L-1}(x_2)\in\mathring{\mathfrak q}_{L-1}$.
Then, by the identity $P_{L-1}S_{L-1}-Q_{L-1}R_{L-1}=1$, we see
that $P_{L-1}(x_2)$ is invertible in the ring
$B_{L-1}/\mathring{\mathfrak q}_{L-1}$. Thus,
$b_L=\bigl(P_{L-1}+a_LQ_{L-1}\bigr)\bigl(x_2\bigr)\times f_L^{-1}$
is invertible in $B_L/\mathring{\mathfrak q}_{L-1}B_L$.
We conclude that $\mathring{\mathfrak q}_L=B_L$, and therefore
$\mathring{\mathfrak q}_n=B_n$. Thus,
$U_v\cap\dot{\mathcal X}(\mathbf Z_w)=\varnothing$, so
$\mathcal X(\mathbf Z_w,(1,n-1))$ does not meet $U_v$ and cannot
contain $\mathcal Y(\mathbf Z_v)$.
\item
Assume that $d_{L-1}=0$ and $(v(L),w(L))=(-,+)$. We note that
$P_K(x_2)=0$ by construction. The identity $P_KS_K-Q_KR_K=1$ then implies
that $Q_K(x_2)$ is invertible in $B_K$, and by assertion~(i)${}_K$
in Lemma~\ref{le:Induc2}, $\widetilde Q_K(0)$
is invertible in $B_K/\mathring{\mathfrak q}_K$. Moreover,
$f_Lb_L=\bigl(P_{L-1}+a_LQ_{L-1}\bigr)\bigl(x_2\bigr)$
belongs to $\mathring{\mathfrak q}_L$. Using assertion~(ii)${}_{L-1}$
in Lemma~\ref{le:Induc2}, we deduce that
$$\bigl(\widetilde P_{L-1}+a_L\widetilde Q_{L-1}\bigr)\bigl(0\bigr)
=\widetilde Q_K(0)(\sigma_{L-1}+a_L)=\widetilde Q_K(0)\sigma_L$$
belongs to $\mathring{\mathfrak q}_L$ too. Therefore $\sigma_L$
belongs to $\mathring{\mathfrak q}_L$, hence to $\mathfrak q$.
However $\sigma_L\notin\mathfrak p$,
because $a_L$ is a summand in the sum that defines $\sigma_L$
whereas $L\notin P(v)$. We must then conclude that
$\mathfrak q\not\subset\mathfrak p$, in other words that
$\mathcal Y(\mathbf Z_v)\not\subset\mathcal X(\mathbf Z_w,(1,n-1))$.
\item
Assume that $d_{L-1}=0$ and $(v(L),w(L))=(+,+)$. As in the
previous case, we note that $\widetilde Q_K(0)$ is invertible in
$B_K/\mathring{\mathfrak q}_K$. But now we have
$f_Lb_L=\bigl(a_LP_{L-1}+Q_{L-1}\bigr)\bigl(x_2\bigr)$, so we get
$$\widetilde Q_K(0)(a_L\sigma_{L-1}+1)\in\mathring{\mathfrak q}_L$$
and then $a_L\sigma_{L-1}+1\in\mathfrak q$. Here however
$a_L\in\mathfrak p$, so $a_L\sigma_{L-1}+1\notin\mathfrak p$.
Again we must conclude that $\mathfrak q\not\subset\mathfrak p$ and
$\mathcal Y(\mathbf Z_v)\not\subset\mathcal X(\mathbf Z_w,(1,n-1))$.
\end{itemize}

Proposition~\ref{pr:IncMulII} is then proved.

\subsection{Loose ends}
\label{ss:LooseEnds}
We can now prove that the MV basis of $V(\varpi)^{\otimes n}$
satisfies the second formula in \eqref{eq:DefYw}. We consider
two words $v$ and $w$ in $\mathscr C_n$ with $w(1)=-$ and
$\wt(v)=\wt(w)$ and look for the coefficient of
$y_v$ in the expansion of $x_-\otimes y_{w'}$ on the MV basis,
where $w'$ is the word $w$ stripped from its first letter.

If $v(1)=-$, then this coefficient is zero except for $v=w$, in
which case the coefficient is one. This follows from Theorem~5.13
in \cite{BaumannGaussentLittelmann}.

If $v(1)=+$, then the path representing $v$ starts above the path
representing $w$. We distinguish two cases.

In the case where $v$ stays strictly above $w$ until the very end,
we can refer to Propositions~\ref{pr:IncMulI} and~\ref{pr:IncMulII}:
the coefficient of $y_v$ is non-zero only if $v$ stays parallel
to $w$ at distance two and the last letter of $w'$ is significant. If this
condition is fulfilled, then the coefficient is one.

In the case where $v$ and $w$ rejoin before the end, after $m$ letters,
then we write $v$ and $w$ as concatenations $+v_{(2)}v_{(3)}$ and
$-w_{(2)}w_{(3)}$, respectively, with $v_{(2)}$ and $w_{(2)}$ of length
$m-1$ and $v_{(3)}$ and $w_{(3)}$ of length $n-m$. By assumption,
$\wt v_{(3)}=\wt w_{(3)}$. We can then apply Proposition~\ref{pr:Trunc}
with $n_1=1$, $n_2=m-1$ and $n_3=n-m$: if $v_{(3)}\neq w_{(3)}$, then
the coefficient of $y_v$ in the expansion of $x_-\otimes y_{w'}$ is zero;
otherwise, it is equal to the coefficient of $y_{+v(2)}$ in the expansion
of $x_-\otimes y_{w(2)}$ on the MV basis of $V(\varpi)^{\otimes m}$.

Thus, Proposition~\ref{pr:Trunc} reduces the second case to the first
one, but for words of length $m$. The coefficient is then non-zero only
if $+v_{(2)}$ stays parallel to $-w_{(2)}$ at distance two and the last
letter of $w_{(2)}$ is significant, in which case the coefficient is
one.

To sum up: if $(v(1),w(1))=(+,-)$, then the coefficient of $y_v$ in the
expansion of $x_-\otimes y_{w'}$ is either zero or one; it is one if and
only if $v$ is obtained by flipping the first letter $-$ of $w$ into a
$+$ and flipping a significant letter $+$ in $w'$ into a $-$. This
shows that the MV basis satisfies the second formula in \eqref{eq:DefYw}.
We have proved:

\begin{theorem}
\label{th:MVbasis}
$(y_w)_{w\in\mathcal C_n}$ is the MV basis of $V(\varpi)^{\otimes n}$.
\end{theorem}

Putting Theorem~\ref{th:MVbasis} alongside Theorem~\ref{th:FrenKhov},
Proposition~\ref{pr:ProjCart}, and Theorem~1.11 in \cite{FrenkelKhovanov},
we obtain the result stated in the introduction.

Pierre Baumann,
Institut de Recherche Mathématique Avancée,
Université de Strasbourg et CNRS,
7 rue René Descartes,
67084 Strasbourg Cedex,
France.\\
\texttt{p.baumann@unistra.fr}
\medskip

Arnaud Demarais,
15 allée du puits,
01290 Crottet,
France.\\
\texttt{arnaud.demarais@ac-dijon.fr}


\begin{thebibliography}{99}
\bibitem{BaumannGaussentLittelmann}
P.~Baumann, S.~Gaussent, P.~Littelmann,
\textit{Bases of tensor products and geometric Satake correspondence},
\href{http://arxiv.org/abs/2009.00042}{arXiv:2009.00042}.

\bibitem{BaumannKamnitzerKnutson}
P.~Baumann, J.~Kamnitzer, A.~Knutson,
\textit{The Mirković-Vilonen basis and Duistermaat-Heckman measures},
to appear in Acta Math.

\bibitem{BeilinsonDrinfeld}
A.~Beilinson, V.~Drinfeld,
\textit{Quantization of Hitchin's integrable system and Hecke eigensheaves},
available at
\href{http://www.math.uchicago.edu/~mitya/langlands.html}
{http://www.math.uchicago.edu/\textasciitilde mitya/langlands.html}.

\bibitem{BravermanGaitsgory}
A.~Braverman, D.~Gaitsgory,
\textit{Crystals via the affine Grassmannian},
Duke Math.\ J.\ \textbf{107} (2001), 561--575.

\bibitem{Demarais}
A.~Demarais,
\textit{Correspondance de Satake géométrique, bases canoniques et
involution de Schützenberger},
PhD thesis, Université de Strasbourg, 2017, available at\\
\href{http://tel.archives-ouvertes.fr/tel-01652887}
{http://tel.archives-ouvertes.fr/tel-01652887}.

\bibitem{FontaineKamnitzerKuperberg}
B.~Fontaine, J.~Kamnitzer, G.~Kuperberg,
\textit{Buildings, spiders, and geometric Satake},
Compos.\ Math.\ \textbf{149} (2013), 1871--1912.

\bibitem{FrenkelKhovanov}
I.~Frenkel, M.~Khovanov,
\textit{Canonical bases in tensor products and graphical calculus for
$U_q(\mathfrak{sl}_2)$},
Duke Math.\ J.\ \textbf{87} (1997), 409--480.

\bibitem{Fulton}
W.~Fulton,
\textit{Intersection theory}, second edition,
Springer, Berlin, 1998.

\bibitem{GaussentLittelmann}
S.~Gaussent, P.~Littelmann,
\textit{LS galleries, the path model, and MV cycles},
Duke Math.\ J.\ \textbf{127} (2005), 35--88.

\bibitem{GoncharovShen}
A.~Goncharov, L.~Shen,
\textit{Geometry of canonical bases and mirror symmetry},
Invent.\ Math.\ \textbf{202} (2015), 487--633.

\bibitem{Haines}
T.~Haines,
\textit{Structure constants for Hecke and representation rings},
Int.\ Math.\ Res.\ Not.\ 2003, no.~39, 2103--2119.

\bibitem{Kashiwara}
M.~Kashiwara,
\textit{Global crystal bases of quantum groups},
Duke Math.\ J.\ \textbf{69} (1993), 455--485.

\bibitem{Lusztig81}
G.~Lusztig,
\textit{Singularities, character formulas, and a $q$-analog of weight
multiplicities},
in \textit{Analyse et topologie sur les espaces singuliers, II, III
(Luminy, 1981)}, pp.\ 208--229, Astérisque \textbf{101--102}, Soc.\
Math.\ France, Paris, 1983.

\bibitem{Lusztig90}
G.~Lusztig,
\textit{Canonical bases arising from quantized enveloping algebras},
J.\ Amer.\ Math.\ Soc.\ \textbf3 (1990), 447--498.

\bibitem{Lusztig93}
G.~Lusztig,
\textit{Introduction to quantum groups},
Progress in Mathematics vol.~110, Birkhäuser Boston, Boston, 1993.

\bibitem{MirkovicVilonen}
I.~Mirković, K.~Vilonen,
\textit{Geometric Langlands duality and representations of algebraic
groups over commutative rings},
Ann.\ of Math.\ \textbf{166} (2007), 95--143.
\end{thebibliography}
\end{document}